\documentclass[11pt]{article}

\usepackage{amsmath,amsfonts,amsthm,enumerate,geometry}

\newcommand{\Om}{\Omega}

\newtheorem{Alemma}{Lemma}[section]
\newtheorem{Atheorem}[Alemma]{Theorem}
\newtheorem{Acorollary}[Alemma]{Corollary}
\newtheorem{Aproposition}[Alemma]{Proposition}
\newtheorem{Aremark}[Alemma]{Remark}
\newtheorem{Adefinition}[Alemma]{Definition}
\newtheorem{assumption}{Assumption}

\def\Xint#1{\mathchoice
{\XXint\displaystyle\textstyle{#1}}%
{\XXint\textstyle\scriptstyle{#1}}%
{\XXint\scriptstyle\scriptscriptstyle{#1}}%
{\XXint\scriptscriptstyle\scriptscriptstyle{#1}}%
\!\int}
\def\XXint#1#2#3{{\setbox0=\hbox{$#1{#2#3}{\int}$ }
\vcenter{\hbox{$#2#3$ }}\kern-.6\wd0}}

\def\dashint{\Xint-}

\def\eps{\varepsilon}

\newcommand{\ou}{{\bar u}}

\newcommand{\R}{{\mathbb R}}
\newcommand{\Per}{{\rm Per}}

\begin{document} 

\title{Front propagation in geometric and phase field models of
  stratified media}

\author{A. Cesaroni\footnote{Dipartimento di Matematica, Universit\`a
    di Padova, Via Trieste 63, 35121 Padova, Italy, email:
    acesar@math.unipd.it} \and C. B. Muratov\footnote{Department of
    Mathematical Sciences, New Jersey Institute of Technology, Newark,
    NJ 07102, USA, email: muratov@njit.edu} \and
  M. Novaga\footnote{Dipartimento di Matematica, Universit\`a di Pisa,
    Largo Bruno Pontecorvo 5, 56127 Pisa, Italy, email:
    novaga@dm.unipi.it} }

\maketitle

\begin{abstract}
  We study front propagation problems for forced mean curvature flows
  and their phase field variants that take place in stratified media,
  i.e., heterogeneous media whose characteristics do not vary in one
  direction.  We consider phase change fronts in infinite cylinders
  whose axis coincides with the symmetry axis of the medium. Using the
  recently developed variational approaches, we provide a convergence
  result relating asymptotic in time front propagation in the diffuse
  interface case to that in the sharp interface case, for suitably
  balanced nonlinearities of Allen-Cahn type. The result is
  established by using arguments in the spirit of
  $\Gamma$-convergence, to obtain a correspondence between the
  minimizers of an exponentially weighted Ginzburg-Landau type
  functional and the minimizers of an exponentially weighted area type
  functional. These minimizers yield the fastest traveling waves
  invading a given stable equilibrium in the respective models and
  determine the asymptotic propagation speeds for front-like initial
  data. We further show that generically these fronts are the
  exponentially stable global attractors for this kind of initial data
  and give sufficient conditions under which complete phase change
  occurs via the formation of the considered fronts.
\end{abstract}

\newpage

\tableofcontents

\section{Introduction}  

Front propagation is a phenomenon ubiquitous to nonlinear systems
governed by reaction-diffusion mechanisms and their analogs, and arises
in many applications, including phase transitions, combustion,
chemical reactions, population dynamics, developmental biology,
etc. There is now a huge literature on the subject dealing with
various aspects of front propagation, from existence of traveling wave
solutions, long time asymptotic behavior, various singular limits in
the presence of small parameters to generalized notions of fronts and
the effects of advection, randomness or stochasticity (see, e.g., the
review in \cite{xin} and references therein).  By a {\em front}, one
usually understands a narrow transition region (interface) in which
the solution of the underlying reaction-diffusion equation changes
abruptly between two equilibria. At the core of the phenomenon of
propagation is the fact that such fronts may exhibit {\em wave-like}
long-time behavior, whereby the level sets of the solution advance in
space with some positive average velocity. This geometric aspect of
the problem also leads to an alternative modeling viewpoint, whereby
fronts are regarded as infinitesimally thin, i.e., as hypersurfaces
whose motion is governed by a geometric evolution law. In the context
of phase field models considered in this paper (see below) the
connection between the {\em diffuse} and {\em sharp} fronts in the
respective diffuse interface and sharp interface models has been the
subject of many studies
\cite{fife,dms,chen,ilmanen93,ess,bss,barles,mugnai11,alfaro12,almat}, starting with the early works
\cite{fife:jcp76,fifemcl,allen79,caginalp86ann,caginalpfife,rubinstein89} 
(these lists of
references are, of course, far from being exhaustive).
  
As a prototypical model, consider the following version of the
Allen-Cahn equation in the presence of a heterogeneous forcing term:
\begin{align}
  \label{rd00}
  \phi_t = \Delta \phi + f(\phi)+ \eps g(\eps x).
\end{align}
Here $\phi = \phi(x, t) \in \R$ is a variable that depends on the
spatial coordinate $x \in \R^n$ and time $t \geq 0$,
$f(\phi)=\phi(1-\phi)(\phi-\frac{1}{2})$ is a balanced bistable
nonlinearity with $\phi = 0$ and $\phi = 1$ being stable equilibria
and $\phi = \frac{1}{2}$ an unstable equilibrium, $g(x)$ is some
sufficiently regular periodic function and $\eps > 0$ is a
parameter.  Such an equation may arise, e.g., in modeling the dynamics
of two co-existing phases in a phase transition with non-conserved
order parameter in a medium with periodically varying properties. When
$\eps \ll 1$, the variations of the properties are weak and slowly
changing in space.  It is then easy to show that in this regime there
exist two uniquely defined equilibrium states (periodic with the same
period as $g$), $v_0$ and $v_1$, with the properties:
\begin{equation}
  \label{vpm}
  v_0(x) = 2 \eps g(\eps x) + O(\eps^2)\qquad v_1
  = 1 + 2 \eps g(\eps x) + O(\eps^2).
\end{equation}
These correspond to the perturbations of the two coexisting phases
$\phi = 0$ and $\phi = 1$, respectively, in the homogeneous Allen-Cahn
equation.  Now define $u = \phi - v_0$. It solves
\begin{equation}
  \label{rd0}
  u_t = \Delta u +f(u) + \eps a(\eps x,u),
\end{equation}
where $a(\eps x,u) = 6 g(\eps x)(u-u^2)+O(\eps)$. This type of
equation for $\eps \ll 1$ and its solutions that invade the $u = 0$
equilibrium will be the main subject of this paper.

On formal asymptotic grounds
\cite{fife,caginalp86ann,rubinstein89,bss,alfaro12,almat}, the dynamics
governed by \eqref{rd0} with some fixed initial condition is expected
to converge as $\eps \to 0$, after rescaling space and time as
\begin{align}
  \label{xt}
  x \to \frac{x}{\eps}, \qquad t \to \frac{t}{\eps^{2}} \,,
\end{align}
to a forced mean curvature flow. More precisely, for each $(x,t)$
fixed the function $u^\eps(x,t) = u(\eps^{-1} x, \eps^{-2}
t)$, where $u(x,t)$ solves \eqref{rd0}, is expected to converge to
either 0 or 1 everywhere except for an $(n-1)$-dimensional evolving
hypersurface $\Gamma(t) \subset \R^n$ separating the regions where $u
= 0$ and $u = 1$ in the limit and whose equation of motion reads
\begin{align}
  \label{mcf}
  V(x) = \frac{g(x)}{c_W} - \kappa(x),
\end{align}
where we used the fact that $g(x) = \lim_{\eps \to 0} \int_0^1
a(x,u)du$.  Here $V(x)$ is the velocity in the direction of the
outward normal (i.e., pointing into the region where $u =0$ in the
limit) at a given point $x \in \Gamma(t)$, $\kappa$ is the sum of the
principal curvatures (positive if the limit set where $u = 1$ is
convex), and
\begin{align}
  \label{aW}
  c_W := \int_0^1 \sqrt{2 W(u)} \, du , \qquad W(u) := -\int_0^u f(s)
  \, ds,
\end{align}
where we defined the double-well potential $W$, associated with $f$,
which is nonnegative and whose only zeros are $u = 0$ and $u =
1$. In view of \eqref{vpm}, the same result would then hold for
$\phi^\eps(x, t) = \phi(\eps^{-1} x, \eps^{-2} t)$, where
$\phi(x,t)$ solves \eqref{rd00}. For well-prepared initial data such a
result was rigorously established by Barles, Soner and Souganidis,
interpreting \eqref{mcf} in the viscosity sense \cite{bss} (for
related results on unforced mean curvature flows see
\cite{chen,dms,ess,ilmanen93}). The case of more general initial
conditions was also treated by Barles and Souganidis in
\cite{barles}. More recently, the problem above was treated within
the varifold formalism by Mugnai and R\"oger under weaker assumptions
on the forcing term and in dimensions two and three
\cite{mugnai11}. Rigorous leading order asymptotic formulas for
solutions of \eqref{rd00} in terms of solutions of \eqref{mcf} were
also recently provided by Alfaro and Matano \cite{alfaro12}.

Note that, since the above mentioned results are local in space and
time, they are not suitable for drawing conclusions about the behavior
as $t \to +\infty$ of solutions of \eqref{rd00} for $\eps \ll 1$, via
the analysis of \eqref{mcf}. Nevertheless, it is widely believed that
\eqref{mcf} should be able to provide information about the long-time
behavior of solutions of \eqref{rd00} for $\eps \ll 1$. For example,
in the context of the Allen-Cahn equation it is interesting to know
how fast the energetically more favored phase invades the
energetically less favored phase following a nucleation event. In the
homogenous setting (i.e., with $g(x) = \bar g > 0$) this would occur
via a radial front moving asymptotically with constant normal
velocity, consistent with \eqref{mcf}
\cite{aronson78,jones83,bronsard91}. In this paper we provide results
for this type of questions for a particular class of heterogeneities
in \eqref{rd00}.

We focus on reaction-diffusion equations and mean curvature flows in
infinite cylinders that describe the so-called {\em stratified
  media}. These are media that are fibered along the cylinder, i.e.,
those whose properties do not change along the cylinder axis, and this
property can be characterized by the dependence of the nonlinearity
for reaction-diffusion equations and of the forcing term for mean
curvature flow only on the transverse coordinate of the cylinder.  By
a cylinder (in the original, unscaled variables), we mean a set
$\Sigma_\eps = \Omega_\eps \times \R \subset \R^n$, where $\Omega_\eps
= \eps^{-1} \Omega$ and $\Omega \subset \R^{n-1}$ is either a bounded
domain with sufficiently smooth boundary or an $(n-1)$-dimensional
parallelogram with periodic boundary conditions, covering the case
discussed earlier in the presence of an axis of translational
symmetry. In the case of a general bounded domain $\Omega$, we also
supplement the problem with homogeneous Neumann boundary
conditions. Our main interest is to provide a convergence result to
relate the asymptotic characteristics of front propagation in the
diffuse interface case with those in the sharp interface case as $t
\to \infty$ when $\eps \ll 1$. More precisely, we wish to characterize
the asymptotic propagation speeds for fronts in $\Sigma_\eps$ invading
the $u = 0$ equilibrium in the case of \eqref{rd0} or, equivalently,
the $\phi = v_0$ equilibrium in the case of \eqref{rd00}, in terms of
uniformly translating graphs solving \eqref{mcf}. We also wish to
characterize the shape of the long time limit of the fronts for
\eqref{rd0} and their relation to those for \eqref{mcf} in the spirit
of $\Gamma$-convergence (as is done for stationary fronts in
\cite{momo}).

\paragraph*{Variational formulation.}

Our methods are essentially variational.  This stems from the basic
observation \cite{mu} (in the one-dimensional setting the idea goes
back to \cite{fifemcl}) that, when the nonlinearity \eqref{rd0} is
translationally invariant along the cylinder axis, the solution of
this equation in the reference frame moving with speed $c > 0$ along
the cylinder $\Sigma_\eps$ may be viewed as a gradient flow in
$L^2_c(\Sigma_\eps) := L^2(\Sigma_\eps; e^{cz} dx)$ generated by the
exponentially weighted Ginzburg-Landau type functional
\begin{align}
  \label{eq:Phic}
  \Phi_c(u) = \int_{\Sigma_\eps} e^{c z} \biggl( \frac12 |\nabla u|^2
  + V(u, y) \biggr) dx.
\end{align}
Here $x = (y, z) \in \Sigma_\eps$, with $y \in \Omega_\eps$ being the
transverse coordinate in the cylinder cross-section and $z \in \R$ the
coordinate along the cylinder axis in the direction of propagation,
and $V(u, y)$ is a suitably chosen potential (see Section
\ref{s:diffuse}).  In particular, traveling wave solutions of
\eqref{rd0} with speed $c$ that belong to the exponentially weighted
Sobolev space $H^1_c(\Sigma_\eps)$, i.e., the space consisting of all
functions in $L^2_c(\Sigma_\eps)$ with first derivatives in
$L^2_c(\Sigma_\eps)$, are fixed points of this gradient flow (see
\cite{lmn,mn1}). In the simplest case of \eqref{rd00} with the
considered cubic nonlinearity and spatially homogeneous forcing $g(x)
= \bar g > 0$, it follows from \cite{mn1} via analysis of \eqref{rd0}
and an explicit computation that for all $\bar g < \sqrt{3} / (36
\eps)$ there exists a unique value of $c^\dag > 0$ satisfying
\begin{align}
  \label{ACTW}
  c^\dag - \frac{8}{9} (c^\dag)^3 = 6 \sqrt{2} \eps \bar g,
\end{align}
and a profile $\bar u \in H^1_{c^\dag}(\Sigma_\eps)$ depending only on
$z$ such that $\bar u$ is the unique (up to translations) minimizer of
$\Phi_{c^\dagger}$ over its natural domain (see below). Furthermore,
by the results of \cite{mn3} the solution of the initial value problem
for \eqref{rd00} with the initial datum in the form of a sharp front:
$\phi(x, 0) = v_0$ for $z > h(y)$ and $\phi(x, 0) = v_1$ for $z \leq
h(y)$, with $h \in C(\overline{\Omega}_\eps)$, converges as $t \to
\infty$ exponentially fast to $v_0 + \bar u$ after a translation by
$R_\infty -c^\dagger t$ for some $R_\infty \in \mathbb R$. Thus, for
every $\eps > 0$ sufficiently small the solution approaches a flat
front perpendicular to the cylinder axis, which, after the rescaling
in \eqref{xt}, moves with the normal velocity
\begin{align}
  \label{VAC}
  V = 6 \sqrt{2} \bar g + O(\eps^2).
\end{align}
This is consistent with the plane wave solution of \eqref{mcf}, in
view of the fact that $c_W = 1 / (6 \sqrt{2})$ according to
\eqref{aW}. A flat front with speed $c^\dag_0 = 6 \sqrt{2} \bar g$ is
also the asymptotic solution of \eqref{mcf} with $g(x) = \bar g$ and
an initial condition in the form of a graph on $\Omega$
\cite{cn}. Furthermore, the corresponding function $\psi(y) = const$
minimizes, for $c = c^\dag_0$, the following exponentially weighted
area type functional
\begin{align}
  \label{Fcbar}
  F_c(\psi) = \int_\Omega e^{c \psi(y)} \left( c_W \sqrt{ 1 + |\nabla
      \psi(y)|^2} - \frac{g(y)}{c} \right) dy,
\end{align}
among all $\psi \in C^1(\overline\Omega)$.  Here $\psi$ defines the
graph $z = \psi(y)$ that represents the sharp interface front.  Note
that the functional $F_c$ has a well-known geometric characterization,
which, however, requires some care \cite{giustibook}.  Let us
introduce the following exponentially weighted perimeter of a set
$S\subseteq \Sigma$:
\begin{align} 
  \label{perc} 
  \Per_c(S,\Sigma) := \sup \Bigg\{ \int_S e^{c z}\left(\nabla \cdot
    \phi + c \hat z \cdot \phi \right) dx : \phi \in
  C^1_c(\Sigma;\R^n), \ |\phi| \le 1 \Bigg\},
\end{align}  
where $\hat z$ denotes the unit vector pointing in the $z$-direction.
Notice that, if the set $S$ has locally finite perimeter, we can write
\begin{align}
  \label{percc}
  \Per_c(S,\Sigma) = \int_{\partial^*S\cap \Sigma} e^{cz}\,d\mathcal
  H^{n-1}(x), 
\end{align}
where $\partial^* S$ denotes the reduced boundary of $S$
\cite{afp,cn}. 

We then define the following geometric functional on measurable sets
$S\subset \Sigma$ with weighted volume $\int_S e^{cz} dx <
+\infty$:
\begin{equation}\label{geo}
  \mathcal{F}_{ c}(S) := c_W \, \Per_{c}(S,\Sigma)-\int_S e^{cz} \,
  g(y) \, dx.  
\end{equation} 
By our assumptions, this functional is indeed well defined for all
such sets.  Since the functionals in \eqref{Fcbar} and \eqref{geo}
agree whenever $S = \{ z < \psi(y) \}$ and $\psi$ is sufficiently
smooth \cite{cn}, in the example in which $g(x) = \bar g > 0$ the sets
$\{z < const\}$ minimize $\mathcal{F}_{ c}$ for $c = c_0^\dag$ over
all such sets. In fact, it is easy to see that they also minimize
$\mathcal F_c$ over its entire domain (see Sec. \ref{sec:tws}).

\paragraph*{Problem formulation and main results.}

The purpose of this paper is to study the long-time behavior of
solutions of \eqref{rd00} or \eqref{rd0} for $\eps \ll 1$ via the
analysis of traveling wave solutions to \eqref{mcf}. In particular we
characterize the asymptotic propagation speed and the shape of the
long time limit of fronts invading the $u = 0$ equilibrium in the case
of \eqref{rd0} or, equivalently, the $\phi = v_0$ equilibrium in the
case of \eqref{rd00}, in terms of uniformly translating graphs solving
\eqref{mcf}.

Throughout the paper we always assume that $\Omega$ is a bounded
domain with a sufficiently smooth boundary (for precise assumptions
see Section \ref{standass}).
%and $2 \leq n\leq 7$ (of course,
%the physically relevant cases correspond to $n = 3$ and $n = 2$)
All the results obtained in this paper remain valid in the periodic
setting, so we do not treat this case separately. We set $\Sigma :=
\Omega \times \R$, and in $\Sigma$ we consider the family of
singularly perturbed reaction-diffusion equations for $u = u(x, t) \in
\mathbb R$, with parameter $\eps>0$ and the space and time rescaled
according to \eqref{xt}:
\begin{equation}\label{rd} 
  \eps u_t=\eps \Delta u+\frac{1}{\eps } f(u)+ a(y,u) \qquad
  (x,t) \in \Sigma\times(0,+\infty) ,
\end{equation} 
with initial datum $u(x,0)=u_0(x) \geq 0$ and Neumann boundary
conditions on $\partial \Sigma$. Here $f(u)$ is a balanced bistable
nonlinearity with $f(0) = f(1) = 0$, and $|a(y, u)| \leq C u$ for some
$C > 0$. For
simplicity, we also assume that $a(x, u)$ does not depend on
$\eps$. Once again, the obtained results remain valid after perturbing
$a$ with terms that can be controlled by $C \eps u$ for some $C > 0$
independent of $\eps$.

As was already mentioned, the singular limit of \eqref{rd} as $\eps\to
0$ was considered in \cite{bss}, where convergence, in a suitable
sense, of positive solutions to the level-set formulation of the mean
curvature flow with a suitable forcing term $g$ was proved. Consider a
family of measurable sets $S(t)\subseteq\Sigma$ with regular boundary,
such that $\Gamma(t)=\partial S(t)$ evolves according to \eqref{mcf}
with
\begin{equation}
  \label{gdef} g(y):= \int_0^1 a(y,s)ds.
\end{equation}
We associate to this flow the following quasilinear parabolic problem
for $h = h(y, t) \in \mathbb R$ in $\Omega$, which corresponds to
\eqref{mcf}:
\begin{equation}\label{fmc} 
  h_t = \sqrt{1+|\nabla h |^2}  \left[\,\nabla \cdot \left(
      \dfrac{\nabla h }{\sqrt{1+|\nabla h|^2}}\right)+\frac{g}{c_W} 
  \right]  \quad {\rm in}\ \Omega \times (0, +\infty)\,,
\end{equation}  
with initial datum $h(y,0)=h_0(y)$, and Neumann boundary conditions on
$\partial \Omega$.  Note that the subgraph $S(t)=\{(y,z)\in\Sigma : z
< h(y,t)\}$ of the solution of \eqref{fmc} coincides with the family
of sets evolving according to \eqref{mcf}, with initial datum
$S_0=\{(y,z)\in\Sigma : z< h_0(y)\}$.

In Section 3, we extend the results of \cite{mn1} on
existence of traveling waves solutions to \eqref{rd} of maximal
propagation speed $c_\eps^\dagger$ to the considered problem for
every $\eps$ sufficiently small, under an assumption on the
forcing term $g$ for the limit problem, which is Assumption
\ref{h4} (see Theorem \ref{maximalspeedrd}).  Moreover, we show that
under a stronger condition on the forcing term $g$, which is
Assumption \ref{h6}, the traveling wave with maximal speed of
propagation is connecting two nondegenerate stable equilibria
(see Proposition \ref{unique}). The nondegeneracy of the equilibria is 
an important property for proving that these waves are 
long-time attractors for solutions to \eqref{rd}.

As for the forced mean curvature flow, we study in Section 4 the
existence of generalized traveling wave solutions, according to
Definition \ref{gentw}, appropriately adapting to the present case the
arguments developed for the periodic case in \cite{cn}. The main
result is Theorem \ref{maximalspeedfmc} which states, under Assumption
\ref{h4}, the existence of a maximal speed of propagation of
generalized traveling waves and provides an accurate description of
the waves traveling at maximal speed.  Moreover, in Theorem
\ref{maxspeedunique} and Theorem \ref{stabat} it is proved that, under
the stronger Assumption \ref{h6}, the traveling waves moving with the
maximal speed are unique and are attractors for the forced mean
curvature flow \eqref{fmc}.

Section 5 contains the main results of the paper. The first one is
Theorem \ref{gammaconvergence}, which provides a convergence result
relating the propagation of diffuse and sharp interfaces. In
particular, we prove that as $\eps\to 0$ the maximal propagation speed
of the traveling waves of \eqref{rd} converges to the maximal speed of
propagation for (generalized) traveling waves of \eqref{fmc} (for some
previous related results see \cite[Proposition 4.3 and Theorem
5.3]{mn2}). By Corollary \ref{r:lead}, the latter is then the average
speed of the leading edge for general front like initial data in the
limit $\eps \to 0$. We also show that as $\eps\to 0$ the traveling
waves of \eqref{rd} moving with the maximal speed converge, up to
translations, to the characteristic function of a set whose boundary
is a traveling wave of \eqref{fmc}, moving with maximal speed.  The
convergence is along subsequences and holds under Assumption \ref{h4}.
Under the stronger Assumption \ref{h6}, we can show that the limit is
independent of the subsequence.  The result is Theorem
\ref{stability}, which states that under Assumption \ref{h6}, the long
time limit of solutions to \eqref{rd} converges, as $\eps\to 0$, to a
traveling wave solution to \eqref{fmc}, translating with maximal speed
$c^\dagger$. In addition, in our proofs we employed some new uniform
estimates for minimizers of Ginzburg-Landau functionals with respect
to compactly supported perturbations, which extend those of
\cite{cc,fv,nv} and are of independent interest. These are presented
in the Appendix.

\smallskip

\paragraph*{Notations.} Throughout the paper $H^1$, $BV$, $L^p$, $C^k$,
$C^k_c$, $C^{k,\alpha}$ denote the usual spaces of Sobolev functions,
functions of bounded variation, Lebesgue functions, continuous
functions with $k$ continuous derivatives, $k$-times continuously
differentiable functions with compact support, continuously
differentiable functions with H\"older-continuous derivatives of order
$k$ for $\alpha \in (0,1)$ (or Lipschitz-continuous when $\alpha =
1$), respectively.  For a point $x \in \Sigma$ in the cylinder $\Sigma
= \Omega \times \mathbb R$, where $\Omega \subset \mathbb R^{n-1}$, we
always write $x = (y, z)$, where $y \in \Omega$ is the transverse
coordinate and $z \in \mathbb R$ is the coordinate along the cylinder
axis. Depending on the context, the symbol $\nabla$ is understood to
denote differential operators acting on functions defined on either
the whole cylinder $\Sigma$, or on its cross-section $\Omega$. The
symbol $B(x, r)$ stands for the open ball in $\mathbb R^n$ with radius
$r$ centered at $x$, and for a set $A$ the symbols $\overline A$,
$|A|$ and $\chi_A$ always denote the closure of $A$, the Lebesgue
measure of $A$ and the characteristic function of $A$,
respectively. We also use the notation $\dashint_A u^2 dx =
\frac{1}{|A|} \int_A u^2 dx$, and the convention that $\ln 0 =
-\infty$ and $e^{-\infty} = 0$.

%%%%%%%%%%%%%%%%%%%%%%%%%%%%%%%%%%%%%%%%%%%%%%%%%%%
%%% standing assumptions 
%%%%%%%%%%%%%%%%%%%%%%%%%%%%%%%%%%%%%%%%%%%%%%%%%%%%%%

\section{Assumptions} \label{standass}

We start by listing the assumptions we shall make on the
nonlinearities $f$ and $a$, and the corresponding forcing $g$
appearing in the evolution problems.  We associate to $f$ and $a$ the
potentials
\begin{equation}\label{pote} 
  W(u) :=-\int_0^u  f(s)\,ds, \qquad G(y,u) := \int_0^u a(y,s)\,ds. 
\end{equation} 
Recall the definition of the forcing term $g$ in \eqref{gdef}:
\begin{equation}
  \label{gdef2} 
  g(y):= G(y,1).
\end{equation}
%We recall that throughout this paper we assume $n\le 7$,
%so that we have better regularity properties for minimizers 
%of the functional $\mathcal F_c$ (see Lemma \ref{den}).

We now state our assumptions on the functions $a$ and $f$.  Assume
that $\partial \Omega$ is of class $C^2$, and let $\alpha\in(0,1]$.

\begin{assumption}\label{g} 
  $a \in C^{\alpha}_{loc}(\overline\Omega \times \R)$, $a_u \in
  C^{\alpha}_{loc}(\overline\Omega \times \R)$, $a(\cdot,0)=0$.
\end{assumption}

\begin{assumption}\label{f}\rm  
  $f \in C^{1,\alpha}_{loc}\big(\R\big)$, $f(0)=f(1)=0$, $f'(0) < 0$,
  $f'(1) < 0$, $W(1) = W(0) = 0$, $W(u) > 0$ for all $u \not= 0,1$,
  and $\displaystyle \liminf_{|u| \to \infty} W(u) > 0$.
\end{assumption}
\noindent Assumption \ref{f} implies that $W(u)$ is a balanced
non-degenerate double-well potential (as a model function one could
think of $W(u) = \frac14 u^2(1-u)^2$ corresponding to the example
considered in the introduction). However, we do not require that $f$
has only one other zero, which is located in $(0,1)$, as is usually
done in the literature. Instead, we only assume that $u = 0$ and $u =
1$ have the same value of $W$, and that $W$ is greater for all other
values of $u$, including at infinity.  Note that by Assumptions
\ref{g} and \ref{f} there exists $C,\delta_0 > 0$, depending only on
$f$ and $a$, such that for every $\eps \leq C^{-1} \delta_0$
\begin{equation}\label{pot}
  \eps^{-1} W(u)- G(y,u) \geq
  0 \qquad \forall (y,u)\in \overline{\Omega}\times 
  \big(\mathbb R \backslash (1-C\sqrt\eps, 1 +C\sqrt\eps)\big),
\end{equation} 
and
\begin{equation}\label{pit}
  \eps^{-1} W(\cdot) - G(y,\cdot) \text{ is increasing on
  }[1+C\eps,1+\delta_0] 
  \qquad\forall y\in \overline{\Omega}.
\end{equation}

\begin{Aremark}\label{max} \rm Observe that, if the initial datum
  $u_0$ satisfies $0\leq u_0(x)\leq 1 + \delta$ for some $\delta \in
  (0, \delta_0)$ and all $x\in \Sigma$, then by the maximum principle
  and \eqref{pit} we have $0\le u(x,t)\le 1+ \delta$ for all $(x,t)\in
  \Sigma\times [0, +\infty)$ and all $\eps \leq C^{-1} \delta$.
\end{Aremark}

We recall the standard definition of the perimeter of a measurable set
$A\subseteq \Omega$ relative to $\Omega$ \cite{giustibook,afp}:
\begin{equation}\label{aper} 
  \Per(A, \Omega) :=    \sup \left\{ \int_A \nabla \cdot \phi (y)  dy:
    \, \phi \in C^1_c(\Omega;\R^{n -1}), \ |\phi|
      \leq 1 \right\}.
\end{equation}  
With the help of \eqref{aper}, the standing assumption to study front
propagation problem for \eqref{rd} and \eqref{fmc} will be the
following condition on the forcing term $g$:

\begin{assumption}\label{h4}\rm 
  Let $g \in C^\alpha(\overline\Omega)$. Then there exists $A\subseteq
  \Omega$ such that
\begin{equation}\label{gcondition} 
  \int_{A} g(y) dy\, >  c_W \,\Per(A,\Omega).
\end{equation} 
\end{assumption} 
\noindent This assumption basically ensures that the trivial state $u
= 0$ is energetically less favorable for $\eps$ sufficiently small,
resulting in the existence of the invasion fronts.

\begin{Aremark}\rm
  Notice that \eqref{gcondition} implies, in particular, that
  $\sup_\Omega g>0$, and is automatically satisfied if
  \begin{equation} \label{g4} 
    \int_\Omega g(y)dy >0.
  \end{equation}
\end{Aremark}

Finally, we list an additional assumption under which stronger
conclusions about the convergence of fronts can be made.
\begin{assumption}
  \label{h6} \rm Let $g \in C^\alpha(\overline\Omega)$ and assume that
  \eqref{g4} holds. Then $\Omega \times \mathbb R$ is the unique
  minimizer of $\mathcal{F}_{c^\dag}$ under compact perturbations
  among sets $S = \omega \times \mathbb R$ with $\omega\subseteq
  \Omega$ and $c^\dag := \inf \{ c > 0: \inf \mathcal F_c \geq 0 \}
  \in (0,\infty)$.
\end{assumption} 
\noindent Clearly, Assumption \ref{h6} is quite implicit.  In
Proposition \ref{p:as4suf} we give some sufficient conditions for it
to hold, first in the two-dimensional case and then in every
dimension.

Throughout the rest of the paper Assumptions \ref{g}--\ref{h4} are
always taken to be satisfied, with $g$ defined by \eqref{gdef2}. The
consequences of Assumption \ref{h6} will be explored in Section
\ref{sec:unique}.

\section{Traveling waves in the diffuse interface
  case} \label{s:diffuse} 

In this section we consider the front propagation problem in the
cylinder $\Sigma$ for the reaction-diffusion equation in \eqref{rd}
with $\eps>0$.  We are particularly interested in the special
solutions of the reaction-diffusion equation \eqref{rd} in the form of
{\em traveling waves}, i.e., solutions of \eqref{rd} of the form $u(x,
t) = \ou(y, z - c t)$, for some $c \in \R$ and $\ou \in C^2(\Sigma)
\cap C^1(\overline\Sigma) \cap L^\infty(\Sigma)$. The constant $c$ is
referred to as the wave speed and the function $\ou$ as the traveling
wave profile.  In particular, the profile of the traveling wave solves
the equation
\begin{equation}\label{twrd} 
  \eps\Delta\ou +c \eps
  \ou_z+\frac{1}{\eps }f(\ou)+a(y,\ou)=0 \qquad   (y,z)\in \Sigma,
\end{equation}
with Neumann boundary conditions ${ \nu}\cdot \nabla{\ou} =0$ on
$\partial \Sigma$.  

More specifically, we are interested in the traveling wave solutions
in the form of fronts invading the equilibrium $v = 0$ from above. By
an equilibrium for \eqref{rd}, we mean a function
$v:\overline{\Omega}\to \R$ which solves
\begin{equation}\label{el} 
  \eps\Delta v+\frac{1}{\eps } f(v )+a(y,v)=0 \qquad y\in \Omega,  
\end{equation} 
with ${ \nu}\cdot \nabla{v} =0$ on $\partial \Omega$. Note that by our
assumptions $v = 0$ is always an equilibrium. In terms of the
propagation speed and the traveling wave profile, front solutions
invading zero from above are bounded solutions of \eqref{twrd} that
satisfy
\begin{equation}\label{ipo}
  c> 0,\qquad \bar u > 0,\qquad {\rm and\ }\bar u(\cdot, z) \to 0\
  {\rm uniformly\ as\ }z \to +\infty. 
\end{equation}

Note that existence and qualitative properties of traveling fronts in
a variety of settings have been extensively studied, starting with the
classical work of Berestycki and Nirenberg \cite{b2}, who analyzed (in
our setting and using our notation) traveling fronts connecting zero
with some equilibrium $v > 0$, i.e., those positive front solutions of
\eqref{twrd} that also satisfy $\ou(\cdot, z) \to v$ uniformly as $z
\to -\infty$. We point out that existence of the considered solutions
is not guaranteed in general. In particular, we have to impose some
condition on $g$ assuring the existence of non-trivial positive
equilibria.  The existence of such non-trivial equilibria for $\eps
\ll 1$ will be a consequence of Assumption \ref{h4} (see Proposition
\ref{proH3}). Similarly, although for a fixed equilibrium $v > 0$ as
the limit at $z = -\infty$ there is at most one (modulo translations)
front solution (see \cite{b2}, under some technical assumptions, and
\cite{mn3} for a general result in the class of the so-called
variational traveling waves), in general front solutions of
\eqref{twrd} may not be unique. There is, however, at most one front
solution of \eqref{twrd} which governs the propagation behavior of
solutions of \eqref{rd} with front-like initial data. These solutions
can be characterized variationally (see \cite{mn1} and the following
section) and are the main subject of our study.

\subsection{Variational principle}

Following the variational approach to front propagation problems
\cite{mn1} (see also \cite{mu,lmn,mn2,mn3}), for every $c > 0$
we associate to the reaction-diffusion equation in \eqref{rd} the
energy functional (for fixed $\eps>0$)
\begin{equation}\label{functionalrd} 
  \Phi_{c}^\eps(u)=\int_\Sigma e^{c  z} \left(\frac{\eps}{2} |\nabla
    u|^2 +\frac{1}{\eps} W(u)- G(y,u)\right) dx.
\end{equation}  
This functional is naturally defined on $H^1_c(\Sigma) \cap
L^\infty(\Sigma)$, where $H^1_{c}(\Sigma)$ is an exponentially
weighted Sobolev space with the norm
\begin{equation} 
\label{expnorm} 
\|u\|^2_{H^1_c(\Sigma)} = \int_\Sigma
e^{cz} (|\nabla u|^2 + |u|^2)\,dx.
\end{equation} 
Furthermore, the functional $\Phi_{ c}^\eps$ is differentiable in
$H^1_c(\Sigma) \cap L^\infty(\Sigma)$, and its critical points
%$u\in H_c^1$ such that $0\leq u<1$  in $\Sigma$  
satisfy the traveling wave equation \eqref{twrd} \cite{lmn,mn1}.

\begin{Aremark}\rm
  Following \cite[Section 4]{b2} (see also \cite[Theorem 4.1]{vega2}
  and \cite[Theorem 3.3(iii)]{mn1}) one can show that every traveling
  wave solution $(c, u)$ of \eqref{rd}, with $c>0$ and satisfying
  \eqref{ipo}, belongs to $H^1_c(\Sigma)$ and is a critical point of
  $\Phi_{ c }^\eps$.  In particular, non-trivial minimizers of
  $\Phi_c^\eps$ are the fastest traveling wave solutions of \eqref{rd}
  invading the equilibrium $u=0$ \cite{lmn,mn1}.  We observe that
  non-trivial critical points of $\Phi_c^\eps$ which are not
  minimizers may also exist and correspond to traveling waves with
  lower speeds (see the discussion in \cite[Section 6]{mn1}).
\end{Aremark}

Critical points of $\Phi_c^\eps$ and, in particular, minimizers of
$\Phi_c^\eps$ play an important role for the long-time behavior of the
solutions of the initial value problem associated with \eqref{rd} in
the case of front-like initial data.  Indeed, in \cite{mn3} it is
proved under generic assumptions on the nonlinearity (see also
\cite{mu,mn1}) that the non-trivial minimizers of $\Phi_c^\eps$ over
$H^1_c(\Sigma)$ are selected as long-time attractors for the initial
value problem associated to \eqref{rd} with front-like initial data.
Also, in \cite{mn1} it was proved under minimal assumptions on the
nonlinearity that the speed of the leading edge of the solution is
determined by the unique value of $c^\dag_\eps > 0$ for which
$\Phi^\eps_{c^\dag_\eps}$ has a non-trivial minimizer.  Appropriate
assumptions to guarantee existence of minimizers of $\Phi^\eps_c$ were
given in \cite{mn1}.  Here we show that in our case these conditions
are verified for every $\eps$ sufficiently small (see Theorem
\ref{maximalspeedrd}).

Let us introduce an auxiliary functional
\begin{equation}\label{functionalaux} 
  E^\eps(v)=\int_\Omega  \left(\frac{\eps}{2} |\nabla v|^2
    +\frac{W(v)}{\eps} - G(y,v)\right)dy\ \qquad v\in H^1(\Om)
  \cap L^\infty(\Om). 
\end{equation} 
Note that $v=0$ is always a critical point of this functional for
every $\eps$.  Moreover, every critical point of $E^\eps$ is an
equilibrium for the reaction-diffusion equation \eqref{rd}.

\begin{Aremark}\rm
  \label{remomo}
  By \cite{momo} (see also \cite{braides}) we have that
\begin{equation}\label{eqE0}
  \Gamma-\lim_{\eps\to 0} E^\eps(u)  = \begin{cases}
    E^0(A) :=c_W\Per(A, \Omega) -\int_{A} g \, dy & {\rm if\
    }u=\chi_A, \\
    +\infty & {\rm otherwise,}
\end{cases}
\end{equation}
where $c_W$ is defined in \eqref{aW}, $A$ is Lebesgue measurable and
the convergence is understood in the sense of $\Gamma$-convergence in
$L^1(\Omega)$ \cite{braides}.
\end{Aremark}

\begin{Adefinition}\label{locmin}
  A function $v\in H^1(\Omega) \cap L^\infty(\Omega)$ is a stable
  critical point of $E^\eps$ if it is a critical point of the
  functional and the second variation of $E^\eps$ is nonnegative, i.e.
  \begin{equation}\label{secondvar}
    \int_\Omega \Big( \eps|\nabla
    \phi|^2 +\left (\eps^{-1} W''(v)-G_{uu}
      (y,v)\right)\phi^2 \Big) dy \geq 0 \qquad \forall
    \phi\in H^1(\Om)\,.
  \end{equation}
  Moreover, $v$ is a nondegenerate stable critical point of $E^\eps$
  if strict inequality holds in \eqref{secondvar}.
\end{Adefinition} 

\begin{Aproposition}\label{proH3}    
  Under Assumptions \ref{g}, \ref{f} and \ref{h4}, there exist
  positive constants $\eps_0$ and $C$ such that for all $\eps <
  \eps_0$ there exists $v_\eps^0 \in H^1(\Omega)$ such that $0 \le
  v_\eps^0\leq 1$ and $E^\eps(v_\eps^0)<0$.
\end{Aproposition} 

\begin{proof}  
  Without loss of generality we may assume that the set $A \subseteq
  \Omega$ in Assumption \ref{h4} has a smooth boundary. Then, by
  Remark \ref{remomo} we have
  \[
  \Gamma-\lim_{\eps\to 0} E^\eps(\chi_A)  = E^0(A)
  \]  
  where $E^0$ is defines in \eqref{eqE0}.  Recalling that $A$ has
  smooth boundary, by the $\Gamma$-limsup construction in \cite{momo}
  (see also Theorem \ref{gammaconvergence} below) there exists a
  family of functions $v_{\eps}^0 \in H^1(\Omega)$, with $0\le
  v_\eps^0 \le 1$ such that $v_\eps^0 \to \chi_A$ in $L^1(\Omega)$ as
  $\eps\to 0$, and
  $$
  \lim_{\eps \to 0}E^{\eps}(v_{\eps}^0) = E^0(A) <0.
  $$ 
  In particular, $E^\eps(v_\eps^0)<0$ for all $\eps$ sufficiently
  small.
 \end{proof} 

By Assumption \ref{f}, $v =0$ is a non-degenerate stable critical
point of the functional $E^\eps$ for every $\eps$ sufficiently
small. Indeed, defining
\begin{equation}\label{nu0}
  \nu_0^\eps := \min_{\int_\Omega \phi^2=1} \,
  \int_\Omega  \Big( \eps |\nabla \phi |^2 + \left( \eps^{-1} W''(0) -
    G_{uu}(y,0)\right) \phi^2 \Big)  dy,
\end{equation} 
observe that there exists $\eps_0>0$, depending on $W''(0)$ and
$\|a_u(\cdot, 0)\|_\infty$, such that
\begin{equation}\label{eqepszero}
  \nu_0^\eps >0 \qquad {\rm for\ all\ }\eps<\eps_0.
\end{equation}

\subsection{Existence of traveling waves}

We now state an existence result for the diffuse interface problem.
This result is a minor modification of the one in \cite[Theorem
3.3]{mn1}.

 \begin{Atheorem} 
   \label{maximalspeedrd} Under Assumptions \ref{g}, \ref{f} and
   \ref{h4}, there exist positive constants $\eps_0$ and $C$,
   depending on $f$, $a$ and $\Omega$, such that for all
   $0<\eps<\eps_0$ there exists a unique $c_\eps^{\dagger} > 0$
   such that
\begin{enumerate}
\item $\Phi^\eps_{c_{\eps}^\dagger}$ admits a non-trivial
  minimizer $\ou_\eps\in H^1_{c_\eps^\dagger} \cap L^\infty(\Sigma)$
  which satisfies
   \begin{equation}\label{lead} 
     \sup \Big\{  z \in \R \ | \sup_{y \in \Omega}  
     \ou_\eps(y, z) > \frac12 \Big\} = 0.  
\end{equation} 
\item $\ou_\eps \in C^2(\Sigma) \cap C^1(\overline\Sigma) \cap
  W^{1,\infty} (\Sigma)$, and $(c^\dagger_\eps, \ou_\eps)$ is a
  traveling wave solution to \eqref{rd}.
\item $0<\ou_\eps\leq 1+C\eps$, $(\ou_\eps)_z<0$ in $\overline\Sigma$,
  and
  \[
  \lim_{z\to +\infty} \ou_\eps(\cdot,z)=0\qquad \lim_{z\to -\infty}
  \ou_\eps(\cdot,z)=v_\eps \qquad {in\ }C^1(\overline\Omega),
  \] 
  where $v_\eps$ is a stable critical point of $E^\eps$ in
  \eqref{functionalaux} with $E^\eps(v_\eps)<0$.
\item $\Phi^\eps_{c_{\eps}^\dagger}(\bar u_\eps) = 0$, and all
  non-trivial minimizers of $\Phi_{c^\dag_\eps}^\eps$ are translates
  of $\bar u_\eps$ along $z$.
  \end{enumerate} 
\end{Atheorem} 

\begin{proof} 
  By \eqref{pot}, \eqref{pit} and Assumption \ref{f} it follows that,
  for $\eps$ sufficiently small, there holds
  $$
  \eps^{-1} W(u)-G(\cdot,u) > 0
  $$
  for all $u < 0$, and
  $$
  \eps^{-1} W(u)-G(\cdot,u) > \eps^{-1}W(1+C\eps)-G(\cdot,1+C\eps)
  $$
  for all $u>1+C\eps$.  By a cutting argument (as in \cite[Theorem
  3.3(i)]{mn1}), we then get, for any $c > 0$ and $u \in H^1_c(\Sigma)
  \cap L^\infty(\Sigma)$, that
  $$
  \Phi_c^\eps(\tilde u)\le \Phi_c^\eps(u),
  $$ 
  where $\tilde u(x) := \max(0, \min (u(x),1+C\eps))$, with strict
  inequality if $\mathrm{ess \, sup}_\Sigma u >1+C\eps$ or
  $\mathrm{ess \, inf}_\Sigma u < 0$. Therefore, the result follows
  from \cite[Theorem 3.9]{mn1} by minimizing $\Phi_c^\eps$ over
  functions with values in $[0, 1 + C \eps]$.  Notice that the
  assumptions in \cite{mn1} are satisfied thanks to Proposition
  \ref{proH3} and \eqref{eqepszero}. The estimate \eqref{lead} is due
  to the fact that, since $\Phi^\eps_{c_{\eps}^\dagger}(\bar u_\eps) =
  0$, by \eqref{pot} we have $\|u_\eps \|_{L^\infty(\Sigma)} >
  \frac12$.  The fact that $(\ou_\eps)_z<0$ up to the boundary of
  $\Sigma$ follows by the strong maximum principle applied to the
  elliptic equation satisfied by $(\ou_\eps)_z$, with Neumann boundary
  conditions, obtained by differentiating \eqref{twrd} in $z$.
  %\re{deleted the rest}
 \end{proof}

\noindent Note that the choice of the value $\frac{1}{2}$ in
\eqref{lead} is arbitrary, and every other value in
$(0,1)$ could be used equivalently.
 
\begin{Aremark}\label{remo}
  \rm Observe that $\Phi_c^\eps(u(y,z-a))=e^{ca}
  \Phi_c^\eps(u(y,z))$ for all $c > 0$, $a\in\R$ and $u \in
    H^1_c(\Sigma) \cap L^\infty(\Sigma)$. In particular, if
  $\ou_\eps$ is a non-trivial minimizer of
  $\Phi_{c^\dag_\eps}^\eps$, then
  $\Phi_{c^\dag_\eps}^\eps(\ou_\eps(y,
    z))=\Phi_{c^\dag_\eps}^\eps(\ou_\eps(y,z-a))=0$.
  Moreover, from \cite[Theorem 3.9]{mn1} we have that
  $\Phi_c^\eps(u)>0$ for every non-zero $u\in H^1_c(\Sigma) \cap
    L^\infty(\Sigma)$ and $c>c_\eps^\dagger$, while $\inf
  \Phi_c^\eps(u)=-\infty$ for $c<c_\eps^\dagger$.
\end{Aremark}

\subsection{Uniform bounds}

We next establish several properties of the traveling wave solutions
$(c^\dag_\eps, \ou_\eps)$ in Theorem \ref{maximalspeedrd} that will
allow us to pass to the limit as $\eps \to 0$ in Section
\ref{secconv}. 

We begin by proving an isoperimetric type inequality for the weighted
perimeter $\Per_c$.

\begin{Aproposition}\label{coronuovo} 
  Let $c>0$ and let $S \subset \Sigma$ be a measurable set with
  $\int_S e^{cz} dx < \infty$. Then
  \begin{align}
    \label{percvc}
    \Per_c(S, \Sigma)\geq c \int_S e^{cz} dx.
  \end{align}
\end{Aproposition} 
\begin{proof} 
  By the definition of the weighted perimeter in \eqref{perc}, we have 
  \begin{align}
    \Per_c(S, \Sigma) \geq \int_S e^{cz} \left(\nabla \cdot \phi + c
      \hat z \cdot \phi \right) dx,
  \end{align}
  for any admissible test function $\phi$. Choosing $\phi(y, z) =
  \eta_\delta (y, z) \hat z$, where the cutoff function
  $\eta_\delta(y, z) := \eta(\delta^{-1} \textrm{dist} (y, \partial
  \Omega)) (1 - \eta(\delta |z| ))$ and $\eta \in C^1(\R)$ with $0
  \leq \eta'(x) \leq 2$ for all $x \in \R$, $\eta(x) = 0$ for all $x
  \leq 1$ and $\eta(x) = 1$ for all $x \geq 2$, for $\delta > 0$
  sufficiently small we find that
  \begin{align}
    \label{etadeltaz}
    \Per_c(S, \Sigma) \geq c \int_S e^{cz} \, \eta_\delta(y, z) \, dx
    - \delta \int_S e^{cz} \eta'(\delta |z|) dx \, .
  \end{align}
  Then, passing to the limit as $\delta \to 0$, we conclude by the
  monotone convergence theorem and the fact that $\int_S e^{cz}
  \eta'(\delta |z|) dx \leq 2 \int_S e^{cz} dx < \infty$.
 \end{proof}

We now prove a uniform upper bound for the
speeds $c^\dag_\eps$ as $\eps \to 0$.

\begin{Aproposition}
  \label{pmaxspeede}
  Let $\eps$ and $c^\dag_\eps$ be as in Theorem
  \ref{maximalspeedrd}. Then there exist constants $\eps_1 > 0$ and $M
  > 0$ depending only on $W$ and $G$ such that $0 < \ou_\eps <
  \frac32$ and $c_\eps^\dagger\leq M$ for all $0 < \eps < \eps_1$.
\end{Aproposition}

\begin{proof}
  By Theorem \ref{maximalspeedrd} and Assumptions 1 and 2 we may take
  $\eps_1 \in (0, \eps_0)$ so small that $0 < \bar u_\eps < \frac32$
  and that $\eps^{-1} W(u) - 2 G(y,u)\geq 0$ whenever $0\leq u\leq
  \frac{3}{4}$. In particular, the set $\{\bar u_\eps >\frac34\}$ has
  positive measure.  Furthermore, we have $W(u) \geq K (1 - u)^2$
  whenever $ u \in [ \frac{1}{2}, \frac32]$, for some $K \in(0,1]$.
  Hence
  \begin{eqnarray*}
    \Phi^\eps_{c^\dag_\eps}(\bar u_\eps) &\geq& \frac{K}{2} \int_{\{
      \bar u_\eps > \frac{1}{2}\}}e^{c^\dag_\eps z} \Big( \eps |\nabla
    \bar u_\eps|^2 + \eps^{-1} (1 - \bar u_\eps)^2 \Big) dx 
    \\
    && - \int_{\{
      \bar u_\eps > \frac34\}}e^{c^\dag_\eps z} G
    (y, \bar u_\eps) \, dx  
    \\
    &\geq& \frac{K}{2} \int_{\{ \bar u_\eps > \frac{1}{2}
      \}}e^{c^\dag_\eps z} | 2 (1 - \bar u_\eps) \nabla \bar u_\eps|
    dx - \int_{\{ \bar u_\eps > \frac34\}}e^{c^\dag_\eps z} G
    (y, \bar u_\eps) \, dx  \\
    &=&\frac{K}{2} \int_{\{ \bar u_\eps > \frac12 \}}e^{c^\dag_\eps z}
    | \nabla h(\bar u_\eps)| dx - \int_{\{ \bar u_\eps >
      \frac34\}}e^{c^\dag_\eps z} G (y, \bar u_\eps) \, dx\, ,
  \end{eqnarray*}
  where $h(u) := \frac14 + (u - 1) |u - 1|$ is an increasing function
  of $u$ with $h(\frac12) = 0$.  In turn, by the co-area formula
  \cite{afp} we get
  \begin{equation}\label{coar2}
    \Phi^\eps_{c^\dag_\eps}(\bar u_\eps)\geq \frac{K}{2}
    \int_0^{\frac12}
    \Per_{c^\dag_\eps}( \{ h(\bar u_\eps) > t \} ) dt -  
    \int_{\{ \bar u_\eps > \frac34\}}e^{c^\dag_\eps z} G (y, \bar
    u_\eps) \, dx . 
  \end{equation} 
  On the other hand, by Proposition \ref{coronuovo} we have
  \[
  \Per_{c^\dag_\eps}( \{ h(\bar u_\eps) >t \} ) \geq c^\dag_\eps
  \int_{ \{ h(\bar u_\eps) >t \} } e^{c^\dag_\eps z} dx,
  \]
  where we noted that by Theorem \ref{maximalspeedrd}(iii) the sets
  $\{ (y, z) : h(\bar u_\eps(y, z)) > t \}$ are bounded above in $z$
  uniformly in $y$ for each $t > 0$ and, hence, Proposition
  \ref{coronuovo} applies.  Therefore, substituting this inequality
  into \eqref{coar2} and using the layer cake theorem yields
  \begin{eqnarray*}
    \Phi^\eps_{c^\dag_\eps}(\bar u_\eps) & \geq & \frac{K c^\dag_\eps}{2}
    \int_0^{\frac12} \int_{ \{ h(\bar u_\eps) >t \} }
    e^{c^\dag_\eps z } dx \, dt - \int_{ \{ \bar u_\eps >
      \frac34\}}e^{c^\dag_\eps z}G(y,\bar u_\eps)dx 
    \\ & \geq & \int_{ \{ \bar u_\eps  > \frac34\}} e^{c^\dag_\eps z
    } \left( \frac12 K c^\dag_\eps h(\bar u_\eps)  - G(y,\bar u_\eps )\right) dx
    \\
    & \geq & \int_{ \{ \bar u_\eps  > \frac34\}} e^{c^\dag_\eps z
    } \left( \frac{3}{32} K c^\dag_\eps  - G(y,\bar u_\eps )\right) dx.
  \end{eqnarray*} 
  We then conclude that $\displaystyle c^\dag_\eps \leq M$, where
  $$
  M := {\frac{32}{3 K}} \max_{(y, u) \in \overline\Omega \times
    [{\frac34}, \frac32]} G(y, u) < + \infty,
  $$ 
  for if not, then from above we have $\Phi^\eps_{c^\dag_\eps}(\bar
  u_\eps )>0$, contradicting the conclusion of Theorem
  \ref{maximalspeedrd}(iv).
 \end{proof}
    
We now state an $\eps$-independent density estimate for minimizers of
$\Phi^\eps_c$.

\begin{Aproposition}\label{densityrd}
  Let $\eps$, $c_\eps^\dagger$ and $\ou_\eps$ be as in Theorem
  \ref{maximalspeedrd}.  Given $\delta\in (0,1)$ and $\bar{x}\in
  \overline\Sigma$, there exist $C, \bar r, \bar \eps>0$ depending
  only on $W$, $G$, $\Omega$ and $\delta$ such that, for every $\eps
  \in (0, \bar\eps)$ and $r\in (\eps, \bar r)$, there holds
  \begin{eqnarray}\label{dens}
    \ou_\eps(\bar x) \geq \delta & \Rightarrow &
    \int_{B(\bar{x}, r)\cap \Sigma} \ou_\eps^2\,dx 
    \geq C r^{n}, \\
    \label{densbis}
    \ou_\eps(\bar x) \leq 1 - \delta & \Rightarrow &
    \int_{B(\bar{x}, r)\cap \Sigma} (1-\ou_\eps)^2\,dx 
        \geq C r^{n}.
  \end{eqnarray}
%  \begin{eqnarray}\label{dens}
%    \ou_\eps(\overline x) \geq \delta & \Rightarrow &
%    \left|\left\{(y,z)\in B(\bar{x}, r)\cap \Sigma : 
%        \ou_\eps(y,z) \geq 
%        \gamma \right\}\right| \geq C r^{n}, \\
%    \label{densbis}
%    \ou_\eps(\overline x) \leq 1 - \delta & \Rightarrow &
%    \left|\left\{(y,z)\in B(\bar{x}, r) \cap \Sigma :
%        \ou_\eps(y,z) \leq 
%        1 - \gamma \right\}\right|\geq C r^{n}.
%  \end{eqnarray}
\end{Aproposition}

\begin{proof}
%  The result follows by an adaptation of the arguments in
%  \cite{nv,fv}, which extend those in \cite{cc}. More precisely,
%  density estimates like \eqref{dens} and \eqref{densbis} have been
%  proved in \cite[Theorem 1.3]{nv} in balls $B(\bar{x}, r)$
%  contained in $\Sigma$. In fact, the assumptions in \cite{nv} are
%  more restrictive than our assumptions here, but the arguments go
%  through with minor modifications, up to considering radii $r$
%  bounded above by a constant $R_0$ depending on $\delta$.  However,
%  since we are interested in uniform estimates up to the boundary of
%  $\Sigma$, which hold true thanks to the Neumann boundary conditions
%  on $\partial\Sigma$, we shall explain how to reduce to the case of
%  interior balls treated in \cite{nv,fv}.
%	
	
	%for local minimizers
  %of the functional
  %$$ 
  %u \mapsto \int_\Sigma \Big( |\nabla u|^2 + F(x,u) \Big) \,dx\,,
  %$$ 
  %where $F$ is an $x$-dependent double-well potential. 
	%However, even if our
        %energy functional $\Phi^\eps_{c_\eps^\dagger}$ satisfies the
        % assumptions in \cite{nv},
	%%by Proposition \ref{pmaxspeede},  
	%in \cite{nv}
	%the estimates
  %are proved only in balls contained in $\Sigma$, while we want to have 
	%

  We only prove \eqref{dens}, since \eqref{densbis} follows
  analogously. For $\eps\in (0,\eps_0)$ we let $\tilde{u}_\eps(x):=
  \bar u_\eps(\eps x)\in C^{1}(\overline{\Sigma}_\eps)$, with
  $\Sigma_\eps=\eps^{-1}\Sigma$.  Then $\tilde{u}_\eps$ is a minimizer
  of the functional
  \begin{equation}\label{eni}
    \Phi_{c^\dag_\eps}(u; B(\bar x,\rho)) := \int_{\Sigma_\eps \cap
      B(\bar x,\rho)}
    e^{{\eps c_\eps^\dagger} z} \left( 
      \frac12|\nabla u|^2+W(u)-\eps G(\eps y,u)\right) dx
  \end{equation}
  for any $\rho > 0$, among functions $u \in C^1
  \big(\overline{\Sigma_\eps \cap B(\bar x, \rho)} \big)$ with fixed
  boundary data on $\partial B(\bar x,\rho) \cap \Sigma_\eps$.
  
  For $\rho > 2$, let us first consider the case in which $B(\bar x,
  \rho) \subset \Sigma_\eps$. Without loss of generality we may assume
  that $\bar{x}=0$.  By our assumptions and standard regularity theory
  \cite{gt}, there exists a constant $M \geq 1$ independent of $\eps$
  and $\rho$ such that
  \begin{equation}\label{eqlip}
    \|\nabla \tilde{u}_\eps\|_{L^\infty(B(0, \rho))}\le M.
  \end{equation}
  Recalling that $\tilde{u}_\eps(0) \geq \delta$, \eqref{eqlip}
  implies that
  \begin{equation}
    \label{lbgrad}
    \dashint_{B(0,R)}\tilde{u}_\eps^2\,dx \ge 
    \frac{1}{|B(0,R)|}\int_{B\left(0,\frac{\delta}{2M}\right)}\frac{\delta^2}{4}\,dx 
    =\frac{\delta^{n+2}}{2^{n+2} R^n M^n}
    \qquad \forall R \in [1,  \rho].
  \end{equation}
  We now note that in the considered case the functional in
  \eqref{eni} satisfies the assumptions of Theorem \ref{p:densL2},
  with all the constants independent of $\eps$ and $\rho$, as long as
  $\rho \leq \eps^{-1}$ and $\eps < \eps_1$, where $\eps_1$ is given
  by Proposition \ref{pmaxspeede}. Therefore, if $r_0 \geq 1$ is the
  integer independent of $\eps$ and $\rho$ defined in Theorem
  \ref{p:densL2}, and
  \begin{align}
    \label{alpha}
    \alpha := \frac{\delta^{n+2}}{2^{n+2} r_0^n M^n} ,
  \end{align}
  we have $0 < \alpha < r_0^{-n} \leq r_0^{1-n}$. Then by Theorem
  \ref{p:densL2} we obtain
  \begin{equation}\label{eqstick}
    \dashint_{B(0,R)}\tilde{u}_\eps^2\,dx \ge
    \alpha  \qquad {\rm for\ all\ } R \in \mathbb N \cap [
    r_0, R_0 ]\ {\rm and\ }\eps<\eps_1, 
  \end{equation}
  provided that $R_0 = \lfloor \eps^{-1} r_1 \rfloor$ satisfies $r_0 +
  1 \leq R_0 < \rho$ and
  \begin{align}
    \label{r1}
    r_1 < \frac{\alpha}{1+ \|G\|_{L^\infty(\Omega \times
        (0,\frac32))}}.
  \end{align}  
  Note that since $\alpha < 1$, this statement is non-empty for all
  $\eps < \eps_2(r_1)$, where $\eps_2(r_1) := \min ( \eps_1, r_1 (2
  + r_0)^{-1})$, and all $\rho$ satisfying $r_1 < \eps \rho \leq
  1$. Furthermore, since $R \geq 1$ in \eqref{eqstick}, extending that
  estimate to an interval yields
  \begin{align}
    \label{desired}
    \dashint_{B(0,R)}\tilde{u}_\eps^2\,dx \geq 2^{-n} \alpha \qquad
    \forall R \in [1, \eps^{-1} r_1],
  \end{align}
  where we also observed that by the definition of $\alpha$ and
  \eqref{lbgrad} the estimate in \eqref{eqstick} holds for all $R \in
  [1, r_0]$ as well.  By a rescaling and a translation, this then
  proves \eqref{dens} for all $\bar x \in \Sigma$ such that
  $\text{dist} (\bar x, \partial \Sigma) > r_1$, for every $r_1 > 0$
  satisfying \eqref{r1} and every $\eps \in (0, \eps_2(r_1))$.

  We now consider the case of $\bar x \in \overline \Sigma$ such that
  $\text{dist} (\bar x, \partial \Sigma) \leq r_1$, for $r_1 > 0$ to
  be fixed momentarily. Once again, assume without loss of generality
  that $\bar x = 0$, which implies that $\text{dist}(0, \partial
  \Sigma_\eps) \leq \eps^{-1} r_1$. Observe that since $\partial\Om$
  is bounded and of class $C^2$, there exists $r_1 > 0$ satisfying
  \eqref{r1} and a diffeomorphism $\phi_\eps:\R^n\to \R^n$ of class
  $C^2$ such that $\phi_\eps(B(0,\rho) \cap \Sigma_\eps)=B^+_\eps$,
  where $B^+_\eps$ denotes the intersection of the ball $B(0,\rho)$
  with the half-space $\{x\in\R^n:\,x_{n}>-{\rm
    dist}(0,\partial\Sigma_\eps)\}$, and $\rho = 2 \eps^{-1} r_1$.  In
  the new coordinate system the functional in \eqref{eni} becomes
  \begin{multline}
    \label{nn} \Phi_{c^\dag_\eps}(u; B(0,\rho)) = \int_{B_\eps^+}
    e^{{\eps c_\eps^\dagger} \hat z \cdot \phi_\eps^{-1}(x')} \Big(
    \,\frac12|D\phi_\eps(x')\nabla \tilde u|^2+W(\tilde u) \\
    -\eps G(\eps \phi_\eps^{-1}(x'), \tilde u) \Big) |\det
    D\phi_\eps(x')|^{-1} \, dx', \qquad
  \end{multline}
  where $\tilde u (x') := u(\phi_\eps^{-1}(x'))$. By a standard
  reflection argument, the minimizer $\tilde{u}_\eps\circ
  \phi_\eps^{-1}$ of \eqref{nn} can be extended from $B^+_\eps$ to a
  minimizer of an energy functional as in \eqref{nn}, but defined on
  the whole of $B(0,\rho)$. Then, since this energy functional still
  satisfies the assumptions of Theorem \ref{p:densL2}, we can repeat
  the arguments from the preceding part of the proof and, possibly
  reducing the value of $r_1$ to some $\bar r > 0$, establish the
  estimate in \eqref{desired} in this case as well, provided that
  $\eps < \bar\eps$ for some $\bar\eps \in (0, \eps_2(\bar r))$.
 \end{proof}

\section{Traveling waves in the sharp interface case} 
\label{sec:tws}

In this section we consider the front propagation problem in the
cylinder $\Sigma$, for the forced mean curvature flow
\eqref{mcf}. Also in this case, we are interested in traveling wave
solutions with positive speed, which are special solutions to the
forced mean curvature flow.

\begin{Adefinition}[Traveling waves] \label{def:tw1}A traveling wave
  for the forced mean curvature flow is a pair $(c,\psi)$, where $c>0$
  is the speed of the wave and the graph of the function $\psi \in
  C^2(\Omega) \cap C^1(\overline\Omega)$ is the profile of the wave,
  such that $h(y,t)= \psi(y)+ c t$ solves \eqref{fmc}.
\end{Adefinition}
\noindent Observe that to prove existence of a traveling wave solution
it is sufficient to determine $c>0$ such that the equation
\begin{equation}\label{equ1} 
  -\,\nabla \cdot \left(\frac{\nabla\psi
    }{\sqrt{1+|\nabla\psi|^2}}\right) =  \frac{1}{c_W}g(y)-   \frac{c
  }{\sqrt{1+|\nabla\psi |^2}}, \qquad y\in\Omega, 
\end{equation} 
with Neumann boundary condition ${\nu}\cdot \nabla\psi =0$ on
$\partial \Omega$, admits a classical solution. The graph of this
solution will be the profile of the traveling wave.

\subsection{Variational principle}

Following the variational approach proposed in \cite[Section 4]{mn2} 
and developed in \cite{cn} (see also \cite{bcn}) for the forced mean curvature
flow, for $c > 0$ we consider the family of functionals $F_c$ defined
in \eqref{Fcbar}.  Note that if $\psi$ is bounded and is a critical
point the functional $F_{{c}}$ then it is a solution to \eqref{equ1}.

After the change of variable $\zeta (y):= \frac{e^{c\psi(y)}}{c} \geq
0$, the functional $F_c$ is equivalent to
\begin{equation}\label{functionalv}
  G_c (\zeta)=\int_\Omega \Big( c_W \sqrt{c^2 \zeta^2(y)+
    |\nabla \zeta(y) |^2} -  g(y) \zeta(y) \Big) dy,
\end{equation}
in the sense that $F_c(\psi) = G_c(\zeta)$ for all $\zeta \in
C^1(\overline\Omega)$ \cite{cn}.  Since the functional $G_c$ is
naturally defined on $BV(\Omega)$ as the lower-semicontinuous
relaxation (see, e.g., \cite{amar08} and references therein), we
introduce the following generalization to the notion of a traveling
wave for \eqref{fmc}. We use the convention that $\ln 0 = -\infty$.

\begin{Adefinition}[Generalized traveling waves] \label{gentw} A
  generalized traveling wave for the forced mean curvature flow
  \eqref{fmc} is a pair $(c, \psi)$, where $c > 0$ is the speed of the
  wave, $\psi = {\frac 1 c} \ln c \zeta$ is the profile of the wave,
  and $\zeta \in BV(\Omega)$ is a non-negative critical point of
  $G_c$, not identically equal to zero.
\end{Adefinition} 

\noindent This definition is consistent with the earlier definition in
the following sense.  Defining $\omega \subseteq \Omega$ to be the
interior of the support of $\zeta$, again, by standard regularity of
minimizers of perimeter type functionals \cite{giustibook,maggi} we
have that $\psi$ solves \eqref{equ1} classically in $\omega$ with $\nu
\cdot \nabla \psi = 0$ on $\partial \Omega \cap \partial \omega$ and,
therefore, we have that $h(y, t) = \psi(y) + ct$ solves \eqref{fmc} in
$\omega$ with Neumann boundary conditions on $\partial \Omega
\cap \partial \omega$. In particular, if $\omega = \Omega$, then the
above definition implies that $(c, \psi)$ is a traveling wave in the
sense of Definition \ref{def:tw1}. In general, however, $\omega$ may
differ from $\Omega$ by a set of positive measure, in which case the
traveling wave profile $\psi$ obeys the following kind of boundary
condition:
\begin{equation}\label{singularbc}
  \lim_{y \to \bar y} \psi(y) = -\infty\qquad \forall \, \bar y
  \in \partial \omega \cap \Omega.
\end{equation} 
In this situation a generalized traveling wave may have the form of
one or several ``fingers'' invading the cylinder from left to right
with speed $c$.
 
The next proposition explains the relation between Assumption \ref{h4}
and the minimization problems associated with functionals $G_c$ and,
hence, $F_c$.
\begin{Aproposition} 
  \label{prop} Let Assumption \ref{h4} hold.  Then there exists a
  unique $c^\dagger> 0$ such that
  \begin{enumerate}[i)]
  \item $\frac{1}{c_W |\Omega|} \int_\Omega g \, dy \leq
    c^\dagger\leq \frac{1}{c_W} \sup_\Omega g$. 
  \item If $0<c<c^\dagger$, then $\inf \{G_c( \zeta)\ : \ \zeta
    \in BV(\Omega) , \ \zeta \geq 0 \}=-\infty$.
  \item If $c>c^\dagger$, then $\inf \{G_c(\zeta)\ : \ \zeta \in
    BV(\Omega) , \ \zeta \geq 0 \}=0$, and $G_c(\zeta)>0$ for every
    non-trivial $\zeta \geq 0$.
   \item If $c = c^\dag$, then there exists a non-trivial $\bar\zeta
     \geq 0$, with $\bar\zeta \in BV(\Omega)$, such that $G_c
     (\bar\zeta) = \inf \{G_c(\zeta)\ : \ \zeta \in BV(\Omega) ,
     \ \zeta \geq 0 \} =0$.
  \end{enumerate}
\end{Aproposition}
  
\begin{proof} The result follows from \cite[Proposition 3.1 and
  Corollary 3.2]{cn} (see also \cite[Proposition 4.1]{mn2}).
 \end{proof}  

Note that the same argument as in \cite[Proposition 3.4, Lemma
3.5]{cn} gives that for all $ \zeta \geq 0 $ such that $ \zeta \in
BV(\Omega)$ we have
\begin{equation}\label{fcpc}
   G_{ c}( \zeta ) =  \mathcal F_{ c}(S_\psi) 
\end{equation}
where $S_\psi=\{(y,z)\in \Omega\times \R\ : \ z< \psi(y)\}$ is
the subgraph of $\psi = {\frac 1 c} \ln c \zeta $.  Moreover if $
\zeta \geq 0 $ is a non trivial minimizer of $G_{c}$, then the
subgraph $S_\psi$ of $\psi$ is a minimizer, under compact
perturbations, of the functional $\mathcal F_c$ defined in
\eqref{geo}. This can be proved as in \cite[Theorem 14.9]{giustibook},
for more details see \cite[Lemma 3.5]{cn}.

We now prove a density estimate for minimizers under compact
perturbations of the functional $\mathcal{F}_c$, which will be useful
in the sequel. Note that related density estimates for the level sets
of minimizers of Allen-Cahn type functionals are proved in Appendix
\ref{sec:appendix}.  Throughout the rest of this section, a set of
locally finite perimeter is identified with its measure theoretic
interior (see \cite{afp}).

\begin{Alemma} \label{den} Given $\bar c > 0$, there exist $r_0>0$ and
  $\lambda>0$ such that for all minimizers $S$ of $\mathcal F_c$ under
  compact perturbations, with $c\in(0,\bar c]$, and for all $\bar x
  \in \overline S$, all $\bar x' \in \overline{\Sigma \backslash S}$
  and all $r\in (0,r_0)$ the following density estimates hold:
  \begin{eqnarray}\label{eqdens}
    |S\cap B (\bar x,r)|&\ge& \lambda\, r^{n},
    \\ \label{eqdens2}
    |(\Sigma \setminus S)\cap B (\bar x',r)|&\ge& \lambda\,
    r^{n}. 
  \end{eqnarray}
  Furthermore, we have $S\subset\Omega\times (-\infty,M]$ for some
  $M\in\R$.
\end{Alemma}

\begin{proof} 
  Let $S$ be a minimizer of $\mathcal F_c$ under compact
  perturbations, $\bar x \in \overline S$ and $r>0$. Up to a
  translation in the $z$-direction, we may assume $\bar x =(\overline
  y ,0)$ for some $\bar y \in \overline \Omega$.  By minimality of $S$
  we have
  \[
  \Per_c (S, \Sigma)-\frac{1}{c_W}\int_{S}e^{cz} g(y)dx \leq
  \Per_c(S\setminus B(\bar x ,r),\Sigma) -\frac{1}{c_W}\int_{
    S\setminus B(\bar x ,r)}e^{cz} g(y)dx.
  \] 
  Therefore, letting $S_r=S \cap B(\bar x ,r)$, we obtain
  \[
  \Per_c(S ,B(\bar x ,r)\cap \Sigma)\leq \int_{\partial B(\bar x
    ,r)\cap S} e^{cz} d\mathcal{H}^{n-1} (x) + \frac{1}{c_W}\int_{S_r}
  e^{cz} g(y) dx.
  \] 
  On the other hand, we have
  \[ 
  \Per_c(S_r,\Sigma)= \Per_c(S ,B(\bar x ,r)\cap
  \Sigma)+\int_{\partial B(\bar x ,r)\cap S} e^{cz} d\mathcal{H}^{n-1}(x),
  \] 
  and hence
  \begin{eqnarray}\nonumber
    \Per_c(S_r,\Sigma)
    &\leq& 2 \int_{\partial B(\bar x ,r)\cap S} e^{cz}
    d\mathcal{H}^{n-1} (x) +
    \frac{1}{c_W}\int_{S_r} e^{cz} g(y) dx 
    \\
    &\leq& 2 \frac{\partial}{\partial r}\left(\int_{S_r} e^{cz}dx\right)
    + \frac{\|g\|_\infty}{c_W}   \int_{S_r} e^{cz}   dx. \label{e}
  \end{eqnarray}   
  Recalling that $\bar x =(\bar y ,0)$, by the relative isoperimetric
  inequality in $\Sigma_0 := \Omega \times (-1 , 1)$ \cite{afp}, there
  exists $C>0$ such that for all $r_0 \leq 1$ and all $r\in (0,r_0)$
  we have
  \begin{eqnarray} \label{iso}
    \Per_c(S_r,\Sigma)&= & \Per_c(S_r,\Sigma_0) \nonumber \\
    & \geq & e^{-\bar c r_0} \Per(S_r,\Sigma_0) \nonumber
    \\
    &\geq& e^{-\bar c r_0} C_{\Sigma_0}|S_r|^{\frac{n-1}{n}} \nonumber \\
    &\geq& C \left(\int_{S_r} e^{cz} dx\right)^{\frac{n-1}{n}}
  \end{eqnarray} 
  for some $C > 0$ depending only on $\bar c$ and $\Omega$.  Let now
  $U(r):=\int_{S_r} e^{cz} dx$, and notice that $\lim_{r\to 0} U(r)=0$
  and that $0 < U(r) \leq e^{\bar c r_0} |S_r|$ for all $r \in
  (0,r_0)$ by our choice of $\bar x$.  {From} \eqref{e} and
  \eqref{iso}, we then get
  \begin{equation} \label{isa} \frac{d U}{dr}(r) \geq C
    U(r)^\frac{n-1}n \qquad {\rm for\ a.e.\ }r\in (0,r_0),
  \end{equation}
  for some $r_0 > 0$ and $C > 0$ depending only on $\bar c$,
  $\|g\|_\infty$, $c_W$ and $\Omega$.

  Estimate \eqref{eqdens} follows from \eqref{isa} by
  integration. Estimate \eqref{eqdens2} follows by the same argument,
  working with $\Sigma \backslash S$ instead of $S$. Finally, the fact
  that $S\subset\Omega\times (-\infty,M]$ for some $M\in\R$ follows
  directly from \eqref{eqdens} and the volume bound $\int_S
  e^{cz}\,dx<+\infty$.
 \end{proof} 

\begin{Aremark} \label{remden}\upshape As a consequence of Lemma
  \ref{den} and the regularity theory for almost minimal surfaces with
  free boundaries (see \cite{gr1,gr2,gr3,giustibook}),
  $\overline{\partial S\cap \Sigma}$ is a hypersurface of class $C^1$
  out of a closed singular set $\Xi_0 \subset\overline \Sigma$ of
  Hausdorff dimension at most $n-8$, and $(\partial S \cap \Sigma)
  \backslash \Xi_0$ is of class $C^2$.
  %(in
  %fact $\Xi_0 \cap\Sigma$ has Hausdorff dimension at most $n-8$). 
  In
  particular, in the physically relevant case $n \leq 3$ the
  hypersurface $\partial S \cap \Sigma$ is of class $C^2$ uniformly in
  $K \times \mathbb R$, for any $K \subset \Omega$ compact (see also
  \cite{taylor}).
\end{Aremark}

In the sequel we need the following lemma, based on the rearrangement
argument in the proof of \cite[Lemma 3.5]{cn}.

\begin{Alemma}
  \label{lrearr}
  Let $c > 0$ and let $S \subset \Sigma$ with $\int_S e^{cz} dx \in
  (0, \infty)$. Then there exists a set $S_\psi = \{ (y, z) \in \Sigma
  : z < \psi(y) \}$ such that $e^{c \psi} \in BV(\Omega)$ and
  \begin{align}
    \label{FcFcpsi}
    \mathcal F_c(S_\psi) \leq \mathcal F_c(S),
  \end{align}
  with strict inequality if $S\not\equiv S_\psi$.
\end{Alemma}

\begin{proof}
  We begin by defining $\psi : \Omega \to [-\infty, \infty)$ as
  \begin{align}
    \label{psiS}
    \psi(y) := \frac 1 c \ln \left( c \int_{S^y} e^{cz} dz \right)
    \qquad \text{for a.e.} \ y \in \Omega,
  \end{align}
  where $S^y := \{ z \in \R : (y, z) \in S \}$ and, as usual, we use
  the convention that $\ln 0 = -\infty$. Notice that if $S_\psi :=
  \{(y, z) \in \Sigma : z < \psi(y) \}$, then by construction $\int_S
  e^{cz} dx = \int_{S_\psi} e^{cz} dx$ and
  \begin{align}
    \label{Spsig}
    \int_S e^{cz} g(y) \, dx = \int_{S_\psi} e^{cz} g(y) \, dx =
    \frac 1 c \int_\Omega e^{c \psi(y)} g(y) \, dy.
  \end{align}
  Now, testing \eqref{perc} with $\phi(y, z) := (\tilde \phi(y) +
  \chi(y) \hat z) \eta_\delta(y, z)$, where $\tilde \phi \in
  C^1_c(\Omega; \R^{n-1})$, $\chi \in C^1_c(\Omega)$, $\delta > 0$ and
  $\eta_\delta$ is as in Proposition \ref{coronuovo}, we have for
  small enough $\delta$ depending on $\tilde\phi$ and $\chi$:
  \begin{multline}
    \int_S e^{cz} (\nabla \cdot \phi + c \hat z \cdot \phi) dx \geq
    \int_\Omega \int_{S^y \cap (-\delta^{-1}, \delta^{-1})} e^{cz} (
    \nabla \cdot \tilde\phi + c \chi) dz \, dy
    \\
    - \int_{S \backslash (\Omega \times (-\delta^{-1}, \delta^{-1}))}
    e^{cz} ( |\nabla \cdot \tilde \phi| + \delta \,
    \eta'(\delta |z|) + c |\chi| ) dx \\
    \geq \int_\Omega \int_{S^y} e^{cz} ( \nabla \cdot \tilde\phi + c
    \chi) dz \, dy - C \int_{S \backslash (\Omega \times
      (-\delta^{-1}, \delta^{-1}))}
    e^{cz} dx \\
    = \int_{S_\psi} e^{cz} ( \nabla \cdot \tilde\phi + c \chi) dx - C
    \int_{S \backslash ( \Omega \times (-\delta^{-1}, \delta^{-1}))}
    e^{cz} dx, \label{perceta}
  \end{multline}
  for some $C > 0$ independent of $\delta$. Observing that the last
  term in the last line of \eqref{perceta} vanishes as $\delta \to 0$,
  we obtain
  \begin{align}
    \label{percphy}
    \Per_c(S, \Sigma) \geq \limsup_{\delta \to 0} \int_S e^{cz}
    (\nabla \cdot \phi + c \hat z \cdot \phi) dx \geq\int_{S_\psi}
    e^{cz} ( \nabla \cdot \tilde\phi + c \chi) dx.
  \end{align}
  In particular, since 
  \begin{align}
    \int_{S_\psi} e^{cz} ( \nabla \cdot \tilde\phi + c \chi) dx =
    \frac{1}{c} \int_\Omega e^{c \psi} (\nabla \cdot \tilde\phi + c
    \chi) d y\,,
  \end{align}
  this implies that $e^{c\psi} \in BV(\Omega)$.

  We claim that taking the supremum in \eqref{percphy} over all
  $\tilde \phi$ and $\chi$ satisfying $|\tilde \phi|^2 + \chi^2 \leq
  1$ yields $\Per_c(S_\psi, \Sigma)$ (for similar arguments, see
  \cite[Theorem 14.6]{giustibook}). Indeed, from \eqref{perceta} with
  $S$ replaced by $S_\psi$ we obtain, after sending $\delta \to 0$ and
  then taking the supremum over all $\tilde\phi$, that
  \begin{eqnarray}\nonumber
    \Per_c(S_\psi, \Sigma) &\geq& \sup_{|\tilde\phi|^2 + \chi^2
      \leq 1 } \int_{S_\psi} 
    e^{cz} ( \nabla \cdot \tilde\phi + c \chi)
    dx 
    \\
    &=:&  \int_\Omega e^{c \psi} \sqrt{ 1 + |\nabla \psi|^2} \, dy\,.
    \qquad \label{perceta2}
  \end{eqnarray}
  We now approximate $\psi$ by smooth functions
  $\psi_\eps$ such that 
  $$
  \lim_{\eps\to 0}\int_\Omega e^{c \psi_\eps}\, dy = \int_\Omega e^{c
    \psi}\,dy
  $$
  and
  $$
  \lim_{\eps\to 0}\int_\Omega e^{c \psi_\eps} \sqrt{ 1 + |\nabla
    \psi_\eps|^2} \, dy = \int_\Omega e^{c \psi} \sqrt{ 1 + |\nabla
    \psi|^2} \, dy\,.
  $$ 
  By the lower semicontinuity of the perimeter functional $\Per_c$, we
  obtain
  \begin{multline}
    \Per_c(S_\psi, \Sigma) \leq \liminf_{\eps \to 0} \,
    \Per_c(S_{\psi_\eps}, \Sigma) \\
    = \liminf_{\eps \to 0} \sup_{\stackrel{\phi \in C^1_c(\Sigma;
        \R^n)}{|\phi| \leq 1}} \int_{S_{\psi_\eps}} e^{cz}
    (\nabla \cdot\phi + c \hat z \cdot
    \phi ) \, dx \\
    = \liminf_{\eps \to 0} \sup_{\stackrel{\phi \in C^1_c(\Sigma;
        \R^n)}{|\phi| \leq 1}} \int_{\partial S_{\psi_\eps}} e^{c
        z} \, \phi \cdot \nu \ d \mathcal H^{n-1}(x),
    \label{perceta3}
  \end{multline}
  where $\nu$ is the normal to $\partial S_{\psi_\eps}$ pointing out
  of $S_{\psi_\eps}$, and we used Gauss-Green theorem to arrive at the
  last line. From this we obtain
  \begin{multline}
    \Per_c(S_\psi, \Sigma) \leq \liminf_{\eps \to 0} 
    \int_{\partial S_{\psi_\eps}} e^{cz} d \mathcal H^{n-1}(x) \\
    = \lim_{\eps \to 0} \int_\Omega e^{c \psi_\eps} \sqrt{ 1 + |\nabla
      \psi_\eps|^2} \, dy = \int_\Omega e^{c \psi} \sqrt{ 1 + |\nabla
      \psi|^2} \, dy\,.
  \end{multline} 
  Therefore, we have, in fact, an equality in \eqref{perceta2}, and
  \eqref{FcFcpsi} then follows by combining
  \eqref{percphy} with \eqref{Spsig}.
  
  Finally we show that the inequality in \eqref{FcFcpsi} is strict if
  $S \not\equiv S_\psi$.
  % In order to prove this, it is enough to
  % show that the first inequality in \eqref{percphy} is strict if $S
  % \not\equiv S_\psi$.
  By Gauss-Green theorem we have
  \begin{multline}
    \label{gsgre}
    \int_S e^{cz} (\nabla \cdot\phi + c \hat z \cdot \phi ) dx =
    \int_{\partial^* S} e^{cz} \phi \cdot \nu \, d \mathcal H^{n-1}(x)
    \\
    \leq \int_{\partial^* S \cap \{ \nu \cdot \hat z \geq -\eps\} }
    e^{cz} \, d \mathcal H^{n-1}(x) + 
    \int_{\partial^* S \cap \{ \nu  \cdot \hat
      z < -\eps\} } e^{cz} \phi \cdot \nu \, d \mathcal H^{n-1}(x),
  \end{multline}
  where $\phi$ is as before, $\nu$ is the unit normal vector to
  $\partial^* S$ pointing out of $S$ and $\eps > 0$ is arbitrary.
  Therefore, for $\chi \geq 0$ we can write
  \begin{multline}
    \int_{S_\psi} e^{cz} ( \nabla \cdot \tilde\phi + c \chi) dx
    \leq
    \limsup_{\delta \to 0} \int_S e^{cz} (\nabla \cdot\phi + c \hat z
    \cdot \phi ) dx \\ 
    \leq \int_{\partial^* S \cap \{ \nu \cdot \hat z \geq -\eps\} }
    e^{cz} \, d \mathcal H^{n-1}(x) + \sqrt{1 - \eps^2}
    \int_{\partial^* S \cap \{ \nu \cdot \hat z < -\eps\} } e^{cz}
    \, d \mathcal H^{n-1}(x) \\
    \leq \int_{\partial^* S} e^{cz} \, d \mathcal H^{n-1}(x) -
    \frac{\eps^2}{2} \int_{\partial^* S \cap \{ \nu \cdot \hat z <
      -\eps\} } e^{cz} \, d \mathcal H^{n-1}(x),
    \label{gsgre2}
  \end{multline}
  where we also used \eqref{percphy}.  We now take the supremum of the
  left-hand side in \eqref{gsgre2} over all $\tilde \phi$ and $\chi$,
  noting that we can restrict $\chi$ to non-negative functions without
  affecting the value of the supremum. Then, reasoning as in the first
  part of the proof, we conclude that the left-hand side of
  \eqref{gsgre2} converges to $\Per_c(S_\psi, \Sigma)$. On the other
  hand, if $S_\psi \not\equiv S$ there exists $\eps > 0$ such that the
  last integral in \eqref{gsgre2} is strictly positive, implying that
  $$
  \Per_c(S_\psi, \Sigma) < \int_{\partial^* S} e^{cz} \, d \mathcal
  H^{n-1}(x).
  $$
  Indeed, it is enough to choose $x\in\partial^* S \cap\{\nu\cdot\hat
  z<0\}$ and let $\eps=-(\nu(x) \cdot\hat z)/2$.  By the properties of
  the reduced boundary \cite{maggi}, it then follows that the set
  $x\in\partial^* S \cap\{\nu\cdot\hat z<-\eps\}$ has positive
  $\mathcal H^{n-1}$ measure.  Combining this with \eqref{Spsig}
  yields the desired result.
 \end{proof}

We summarize all the conclusions above into the following
proposition connecting the non-trivial minimizers of $G_c$ with those
of $\mathcal F_c$ on its natural domain, i.e., among all measurable
sets $S \subset \Sigma$ with $\int_S e^{cz} dx < \infty$.

\begin{Aproposition}
  \label{pminFc}
  Let Assumption \ref{h4} hold, and let $c^\dag$ be as in Proposition
  \ref{prop}. Then
  \begin{enumerate}
  \item[i)] If $0 < c < c^\dag$, then $\inf \mathcal F_c = -\infty$.

  \item[ii)] If $c > c^\dag$, then $\mathcal F_c (S) > 0$ for all $S
    \subset \Sigma$ with $\int_S e^{cz} dx > 0$.

  \item[iii)] There exists a non-trivial minimizer of $\mathcal
    F_{c^\dag}$, and $\mathcal F_{c^\dag}(S) = 0$. Furthermore, $S$ is
    a non-trivial minimizer of $\mathcal F_{c^\dag}$ if and only if $S
    = \{(y, z) \in \Sigma: z < \psi(y) \}$, where $\psi = \frac{1}{
      c^\dag} \ln c^\dag \zeta$ and $\zeta \geq 0$ is a non-trivial
    minimizer of $G_{c^\dag}$.
  \end{enumerate}
\end{Aproposition}

\begin{proof}
  i) follows from \eqref{fcpc} and Proposition
  \ref{prop} (ii). To prove ii), we note that in view Proposition
  \ref{prop} (iii) and \eqref{fcpc} we have $\mathcal F_c(S_\psi) > 0$
  for all $S_\psi$ as in Lemma \ref{lrearr} and apply
  \eqref{FcFcpsi}. Finally, 
  iii)  follows from \eqref{fcpc}, Proposition
  \ref{prop} (iv) and Lemma \ref{lrearr}.
 \end{proof}

\subsection{Existence, uniqueness and stability of generalized
  traveling waves}
\label{sec:unique}

The content of this section is a slight extension of \cite[Section
3]{cn}. We provide proofs of the results for the reader's convenience.

The characterization of minimizers of the geometric functional
$\mathcal F_c$ in Proposition \ref{pminFc} yields the following
existence result for generalized traveling waves. 

\begin{Atheorem}[Existence of generalized traveling
  waves]\label{maximalspeedfmc}
  Let Assumption \ref{h4} hold. Then there exists a unique $c^\dag >
  0$, which coincides with the one in Proposition \ref{prop}, such that:
  \begin{enumerate}[i)]
  \item There exist a function $\psi: \Omega \to [-\infty, \infty)$
    such that $(c^\dagger, \psi)$ is a generalized traveling wave for
    the forced mean curvature flow and the set $S_\psi := \{ (y, z)
    \in \Sigma \ | \ z < \psi(y) \}$ is a minimizer of $\mathcal
    F_{c^\dag}$.

  \item The set $ \omega :={\{\psi> -\infty\}}$ is open and satisfies
    $E^0(\omega) < 0$, where $E^0$ is defined in
    \eqref{eqE0}. Moreover, $\omega \times \R$ is a minimizer of
    $\mathcal F_{c^\dag}$ under compact perturbations, and $\psi\in
    C^2(\omega)$.

  \item $\psi$ is unique up to additive constants on every connected
    component of $\omega$, in the following sense: there exists a
    number $k \in \mathbb N$ and functions $\psi_i : \Omega \to
    [-\infty, \infty)$ for each $i = 1, \ldots, k$ such that $\omega_i
    := {\{\psi_i > -\infty\}} \ \not= \emptyset$ are open, connected
    and disjoint, $\psi_i \in C^2(\omega_i)$ and $\psi = \ln \left(
      \sum_{i=1}^k e^{\psi_i + k_i} \right), $ for some $k_i \in
    [-\infty, \infty)$.

  \item %$\partial S_\psi$ is a hypersurface of class $C^2$ in $\overline{\Sigma}$.
    There exists a closed singular set $\sigma\subset \overline{\partial\omega
      \cap \Omega}$ of Hausdorff dimension at most $n-9$ 
    %(and $\sigma\cap \Omega$ has
    %Hausdorff dimension at most $n-9$) 
    such that $\psi \in C^1(\overline \omega \backslash \sigma)$ and
    $\partial \omega\setminus\sigma$ is a $C^2$ solution to the
    prescribed curvature problem
    \begin{equation}\label{pc}c_W \kappa  = \, g \qquad \text{on } 
      (\partial\omega\cap \Omega )\setminus \sigma,
    \end{equation}
    where $\kappa$ is the sum of the principal curvatures of
    $(\partial\omega\cap \Omega)\setminus\sigma$, with Neumann
    boundary conditions $\nu_{\partial \omega }\cdot \nu_{\partial
      \Omega}=0$ at $(\overline{\partial\omega \cap \Omega}
    \cap \partial \Omega ) \setminus\sigma$.
\end{enumerate}
\end{Atheorem}  

\begin{proof} 
  We recall that by Proposition \ref{pminFc} (iii) there exists a set
  $S_\psi$ which is a non-trivial minimizer of
  $\mathcal{F}_{c^\dagger}$, and $\mathcal{F}_{c^\dagger}(S_\psi)=0$.
  % Furthermore, by Remark \ref{remden} the hypersurface $\partial
  % S_\psi \cap \Sigma$ is of class $C^2$ out of a singular set
  % $\Xi_0$.
  Observe also that the class of minimizers of
  $\mathcal{F}_{c^\dagger}$ with respect to compact perturbations is
  invariant with respect to shifts along $z$ and is closed with
  respect to the $L^1_{loc}$ convergence of their characteristic
  functions \cite{giustibook}. Therefore, translating the minimizer
  $S_\psi$ towards $z = +\infty$ and passing to the limit,
  % as in Lemma \ref{den}
  we get that $\omega \times \R$ is a minimizer under compact
  perturbations of $\mathcal{F}_{c^\dagger}$.  The regularity of
  $\partial (\omega\times\R)=\partial\omega\times\R$ is then a
  consequence of the classical regularity theory for minimal surfaces
  with prescribed mean curvature in bounded domains (see
  \cite{giustibook,gr1,gr2,gr3}).  In particular \eqref{pc}, with Neumann
  boundary conditions, follows from the Euler-Lagrange equation for
  $\mathcal{F}_{c^\dagger}$, observing that $\nu_{\partial
    \omega\times\R}\cdot\hat z=0$.  The inequality $E^0(\omega) < 0$
  follows from \cite[Remark 3.12]{cn}.
   
  {From} the density estimate \eqref{eqdens}, reasoning as in
  \cite[Theorem 14.10]{giustibook}, we derive that $\psi$ is bounded
  above in $\omega$.  Moreover, reasoning as in \cite[Theorem
  14.13]{giustibook} (see also \cite[Proposition 3.7]{cn}), we obtain
  that $\psi$ is regular on $\overline \omega \backslash \sigma
  \supseteq \omega$, in particular $\psi\in C^2(\omega) \cap
  C^1(\overline \omega \backslash \sigma)$. From this, we get that
  $\psi$ is a solution of \eqref{equ1} in $\omega$ with
  $c=c^\dagger$. Moreover $\psi$ satisfies the Neumann boundary
  conditions on $\partial \omega\cap \partial \Omega$
  \cite{giustibook}.
  % Indeed if there were $x_n\to
  % x\in \partial \Omega$ such that $\psi(x_n)\to -\infty$, we would
  % get a contradiction to the density estimate \eqref{eqdens2}
  % applied to $\Sigma\setminus S_\psi$ at $x_n$ for $n$ sufficiently
  % large.
  The fact that $\psi$ is uniquely defined up to translations on every
  connected component of $\omega$ follows from Proposition
  \ref{pminFc}(iii) and the convexity of $ G_{c^\dagger}$ (see
  \cite[Propositions 3.7 and 3.10]{cn}).
 \end{proof} 
  
% From the regularity of $\psi$ we derive further regularity of
% $S_\psi$, in particular it is a $C^2$ hypersurface.
% Moreover $S_\psi$ is a $C^2$ hypersurface uniformly in $\Sigma$. We
% prove that $S_\psi$ is a $C^1$ hypersurface uniformly in $\Sigma$,
% and from this, by standard elliptic regularity theory, recalling
% that the boundary of $\Sigma$ is $C^{2,1}$ and $g\in C^\alpha$ we
% obtain the uniform $C^2$ regularity.  Assume by contradiction that
% there exists a sequence $x_k=(y_k,z_k)\in \partial S_\psi$ such that
% $\partial S_\psi$ is not of class $C^1$ in a uniform neighborhood of
% $x_k$, as $k\to+\infty$. Up to extracting a subsequence we can
% assume that $z_k\to -\infty$ and $y_k\to \tilde y$, as
% $k\to+\infty$, for some $\tilde y\in\overline\Omega$.  Since the
% class of mimizers of $\mathcal F_c$ is invariant for vertical
% translations and it is compact (see \cite{giustibook,maggi}),
% passing to the limit in the translated minimizers $S_k=S_\psi-z_k$
% we obtain a minimizer $\widetilde S$ of $\mathcal F_c$, under
% compact perturbations, such that $\partial \widetilde S$ is not
% $C^1$ in a neighborhood of $\tilde x=(\tilde
% y,0)\in\partial\widetilde S$, thus reaching a contradiction.

\begin{Aremark}\upshape We observe that if $\psi$ is a regular 
  solution to \eqref{equ1} in some set $\omega \subseteq \Omega$, such
  that $\psi$ is bounded from above, $\psi(y) \to -\infty$ as $y
  \to \partial \omega \cap \Omega$, and $\psi = -\infty$ on $\Omega
  \backslash \omega$, then $\zeta=\frac{e^{c\psi}}{c}\in BV(\Omega)$
  and $G_c(\zeta)=0$.  This, in particular, implies that $c\le
  c^\dagger$. Indeed, if $c>c^\dagger$, $\zeta=\frac{e^{c\psi}}{c}$
  would be a non trivial minimizer of $G_c$, contradicting Proposition
  \ref{prop}(iii).  This means that the variational method selects the
  {\it fastest} generalized traveling waves which are bounded from
  above.
\end{Aremark}

%\subsection{Uniqueness and stability of traveling
%  waves} \label{sec:unique} 

Under Assumption \ref{h6}, which is considerably stronger than
Assumption 3, we can prove uniqueness and stability of traveling waves
for the mean curvature flow. We begin by giving several sufficient
conditions that lead to Assumption \ref{h6}. 
%\re{are you sure it's
%  $\Omega$ and not $\overline\Omega$ everywhere below?}

\begin{Aproposition} 
  \label{p:as4suf}
  Let \eqref{g4} hold, and let $C_\Omega$ be the relative
  isoperimetric constant of $\Omega$.  Then Assumption \ref{h6} holds
  if one of the following conditions is verified:
\begin{enumerate}[i)]
\item there is no embedded hypersurface $\partial \omega
  \subseteq \Omega$ which solves the prescribed curvature problem
  \begin{equation}\label{pres}
  c_W \kappa =g
  \qquad \partial \omega\cap \Omega \,,
  \end{equation}
  with Neumann boundary conditions
 $\nu_{\partial \omega}\cdot \nu_{\partial \Omega}=0$ 	on
  $\partial \omega\cap\partial\Omega$,
\item $n=2$ and $g>0$ on $\overline\Omega$,
\item $\min_{\overline{\Omega}} g\leq 0$ and $\max_{\overline{\Omega}}
  g-\min_{\overline{\Omega}} g< C_\Omega c_W
  2^{\frac{1}{n-1}}|\Omega|^{-\frac{1}{n-1}}$,
\item $n>2$, $g>0$ on $\overline\Omega$ and $\max_{\overline{\Omega}}
  g < C_\Omega c_W 2^\frac{1}{n-1}|\Omega|^{-\frac{1}{n-1}}$,
\item $n>1$, $g\in C^1(\overline\Omega)$, $g>0$ on $\overline\Omega$,
  and $\min_{\overline{\Omega}}\, (g^2-(n-1)|\nabla g|)>0$.
\end{enumerate} 
\end{Aproposition}

\begin{proof} 
  $(i)$ follows from Theorem \ref{maximalspeedfmc}. Indeed, if
  $\omega$ is as in Theorem \ref{maximalspeedfmc}, then
  $\partial\omega$ is a solution of the prescribed curvature problem
  \eqref{pres}. $(ii)$ comes from $(i)$, observing that if $\omega$ is
  a solution of the prescribed curvature problem in $\R$, then $g=0$
  on $\partial \omega$.  The proof of $(iii)$ and $(iv)$ is given in
  \cite[Proposition 3.16]{cn} (see also \cite[Proposition 4.6]{CMN2}),
  and $(v)$ is proved in \cite{cls}.
 \end{proof}

We now state an existence and uniqueness result. Note that, in view of
Proposition \ref{pminFc}, the value of $c^\dag$ in Assumption \ref{h6}
coincides with that in Proposition \ref{prop}.

\begin{Atheorem}[Existence and uniqueness of traveling
  waves] \label{maxspeedunique} Under Assumption \ref{h6}, there exist
  a unique $c^\dag > 0$ and a unique $\psi \in C^2(\Omega) \cap
  C^1(\overline\Omega)$ such that $\max_{y \in \overline\Omega}
  \psi(y) = 0$, and $(c^\dagger, \psi)$ is a traveling wave for the
  forced mean curvature flow \eqref{fmc}.  Moreover, $\psi$ is the
  unique minimizer of the functional $F_{c^\dagger}$ over
  $C^1(\overline\Omega)$, up to additive constants, and $S = \{ (y, z)
  \in \Sigma : z < \psi(y) \}$ is the unique minimizer of $\mathcal
  F_{c^\dag}$ up to translations in $z$.
\end{Atheorem}
 
\begin{proof} 
  From Theorem \ref{maximalspeedfmc} we get the existence of a
  generalized variational traveling wave $(c^\dagger, \psi)$ with
  $\omega \subseteq \Omega$, and $\omega\times\R$
  is a minimizer of $\mathcal{F}_{c^\dagger}$ under compact
  perturbations. Then by Assumption \ref{h6} necessarily $\omega =
  \Omega$ and, hence, $\psi > -\infty$ in $\Omega$.

  We now claim that $\psi \geq M$ in $\Omega$ for some $M \in \R$.
  Assume by contradiction that there exists $x_n\to x\in \partial
  \Omega$ such that $\psi(x_n)\to -\infty$. By construction we have
  that the subgraph $S_\psi$ of $\psi$ is a minimizer of
  $\mathcal{F}_{c^\dagger}$. So, we can apply the density estimate
  \eqref{eqdens2} to $\Sigma\setminus S_\psi$ at $x_n$ and obtain a
  contradiction, if $n$ is sufficiently large.

  Since $\psi$ is a bounded regular minimizer of $F_{c^\dagger}$, it
  satisfies the Neumann boundary conditions on $\partial
  \Omega$. Moreover, by standard elliptic regularity theory $\psi \in
  C^2(\Omega) \cap C^1(\overline\Omega)$.  Finally, uniqueness of the
  pair $(c^\dagger, \psi)$ is a consequence of the strong maximum
  principle. Indeed, if there are two smooth solutions $(c^\dagger_1,
  \psi_1)$ and $(c^\dagger_2, \psi_2)$ to \eqref{equ1} with $c^\dag_2
  > c^\dag_1$, then by a suitable translation we may assume that
  $\psi_2 < \psi_1$. Then, using those functions as initial data for
  \eqref{fmc}, we find that the solutions of \eqref{fmc} touch at some
  $t > 0$, contradicting the comparison principle for \eqref{fmc}
  \cite{pw}. If, on the other hand, $c^\dag_1 = c^\dag_2$, again, by a
  suitable translation the two solutions can be made to touch at a
  point, while $\psi_2 \leq \psi_1$. Then by strong maximum principle
  for \eqref{equ1} we have $\psi_1 = \psi_2$ \cite{pw}.
 \end{proof}  

Moreover, we get the following stability result.

\begin{Atheorem} \label{stabat} Let Assumption \ref{h6} hold, let
    $(c^\dag, \psi)$ be as in Theorem \ref{maxspeedunique}, and let
  $h(y,t)$ be the unique solution to \eqref{fmc} with Neumann boundary
  conditions and initial datum $h(y,0)=h_0(y)\in
  C(\overline{\Omega})$.  Then there exists a constant $k\in\R$ such
  that for all $\beta \in (0, 1)$ there holds
  \begin{equation*}
    h(\cdot,t)-c^\dagger t - k \longrightarrow \psi \qquad {\rm in\ }
    C^{1,\beta}(\overline\Omega),{\rm\ as\ }t\to+\infty.
  \end{equation*}
\end{Atheorem} 
\begin{proof} 
  The proof can be obtained by a straightforward adaptation of the
  argument in \cite[Corollary 4.9]{cn}.
 \end{proof}
 
\begin{Aremark}\label{remarkunique} \upshape  If we assume a weaker
  assumption than Assumption \ref{h6}, i.e. that there is at most one
  set $\omega'\subseteq \Omega$ such that $\omega'\times \R$ is a
  minimizer under compact perturbations of the geometric functional
  $\mathcal{F}_{c^\dag}$, then we can prove an analogue of the
  previous stability result.  Indeed under this assumption there
  exists a unique (up to additive constants) generalized traveling
  wave $(c^\dagger,\psi)$, and $\psi$ is supported on $\omega$, where
  $\omega$ can be either $\omega'$ or the whole $\Omega$.  Moreover,
  there exists a constant $k\in\R$ such that, as $t\to+\infty$ there
  holds  
\begin{equation*}  
  h(\cdot,t)-c^\dagger t - k\longrightarrow 
    \begin{cases} 
      \psi & {\rm in\ } C^{1,\beta}_{loc}(\omega),\\ -\infty & {\rm
        locally \ uniformly \ in \ } \Omega\setminus
      \omega, 
    \end{cases} 
\end{equation*} 
for every $\beta \in (0,1)$, where $h(y,t)$ is the unique solution to
\eqref{fmc} with Neumann boundary conditions, and initial datum
$h(\cdot,0)=h_0 \in C(\overline{\Omega})$.  For the proof, see
\cite[Theorem 4.7 and Remark 4.8]{cn}.
\end{Aremark} 

\section{Asymptotic behavior as $\eps\to 0$}\label{secconv}

In this section we prove the main result of this paper, namely, the
convergence result, as $\eps\to 0$, of the traveling waves of
\eqref{rd} to the generalized traveling waves for \eqref{fmc}. For $M
> 0$ and $k > 0$, let us introduce the notations that will be used
throughout the proofs in this section:
\begin{align}
  \label{SigMGk}
  \Sigma_M := \Omega \times (-M, M), \qquad \ \|G \|_{k,\infty}=
  \max_{u\in [0,k], y\in\overline{\Omega}}|G(y,u)|.
\end{align}
We begin with the following basic compactness result.

\begin{Alemma} \label{coer} Let $c>0$ and $c_\eps\to c$ as $\eps \to
  0$. Let $u_\eps\in H^1_{c_\eps} (\Sigma) $ be such that $0\leq
  u_\eps\leq k$ and $\Phi^\eps_{c_\eps } (u_\eps)\leq K$, for some $k,
  K>0$ independent of $\eps$.  Assume that:
  \begin{equation}\label{i} \text{ $\exists \delta\in (0,1)$ such
      that $u_\eps(y,z)\leq \delta $ for all $y\in\Omega$ and $z> 0$.  }
  \end{equation}
  Then there exists $ u\in BV_{loc}(\Sigma,\{0,1\})$ such that
  $u(\cdot, z) = 0$ for all $z > 0$ and 
 \[
 u_\eps\to u\qquad \text{ in }L^1_{loc}(\Sigma),
 \] 
 upon extraction of a sequence. 
 \end{Alemma} 

\begin{proof}
  Recall that by \eqref{pot} there exists $\eps_\delta > 0$ such that
  for every $\eps<\eps_\delta$ the integrand in \eqref{functionalrd}
  is positive for all $z \geq 0$.  Then, by our assumptions, 
  for every $M > 0$ we get
\begin{eqnarray*} 
  K&\geq& \Phi^\eps_{c_\eps } (u_\eps)\geq \int_{\Sigma_M} e^{c_\eps
    z} \frac{W(u_\eps)}{\eps} dx - \int_{-\infty}^M \int_\Omega
  e^{c_\eps z}  G(y,u_\eps) dy \, dz 
  \\
  &\geq& \int_{\Sigma_M} e^{c_\eps z} \frac{W(u_\eps)}{\eps} dx
  -\frac{|\Omega|}{c_\eps} \|G \|_{k,\infty}  e^{c_\eps M} .
\end{eqnarray*}
Therefore, for every $M > 0$ and $\eps$ small enough we have
\begin{equation}\label{im}\int_{\Sigma_M} e^{c_\eps z} W(u_\eps) dx
  \leq \eps K\left(1 +\frac{2 |\Omega|}{c K}\|G \|_{k,\infty}
    e^{c M}\right).
\end{equation}
On the other hand, by our assumptions and the Modica-Mortola trick
\cite{momo}
\begin{eqnarray*}   
  & & \int_{\Sigma_M}e^{c_\eps z} |\nabla u_\eps| \sqrt{2W(u_\eps )}dx \\ 
  &\leq &    \int_{\Sigma_M}  e^{c_\eps z}\left[\left(\sqrt{\frac{
          \eps }{2}}|\nabla u_\eps|
      -\sqrt{\frac{W(u_\eps)}{\eps}}\right)^2+|\nabla u_\eps|
    \sqrt{2W(u_\eps )}\right]dx  \\  
  &=& \int_{\Sigma_M}  e^{c_\eps z}\left(\frac{\eps}{2}|\nabla
    u_\eps|^2 +\frac{W(u_\eps)}{\eps}\right)dx\\ 
  &\leq & \Phi^\eps_{c_\eps} (u_\eps)
  +\frac{|\Omega|}{c_\eps}\|G \|_{k,\infty}e^{c_\eps M}     
  \leq   K\left(1+\frac{2 |\Omega|}{c K}\|G \|_{k,\infty} \, e^{c M} \right).
\end{eqnarray*}

We define
\begin{equation}
  \label{phi} \phi(u) := \int_0^u \sqrt{2W(s)} ds,
\end{equation} 
and rewrite the previous inequality as
\begin{equation}\label{iim} 
  \int_{\Sigma_M} e^{c_\eps z} |\nabla \phi(u_\eps)| dx  \leq
  K\left(1+\frac{2 |\Omega|}{c K}\|G\|_{k,\infty}\,
    e^{c M}\right).  
\end{equation}
This implies that $\phi(u_\eps)$ are uniformly bounded in
$BV(\Sigma_M)$ for every $M > 0$. By compactness theorem in $BV$ (see
\cite{afp}), we then get that upon extraction of a sequence
$\phi(u_\eps)$ converges in $L^1_{loc} (\Sigma)$ to a function $w \in
BV_{loc}(\Sigma)$. Therefore, since $u \mapsto \phi(u)$ is a
continuous one-to-one map, this implies that up to a subsequence
$u_\eps$ converges to $u=\phi^{-1}(w)$ almost everywhere and in
$L^1_{loc}(\Sigma)$. Furthermore, by \eqref{im} $u$ takes values in
$\{0,1\}$ and, hence, $u=c_W^{-1} w$ almost everywhere. Therefore, $u
\in BV_{loc} (\Sigma; \{0, 1\})$.

Eventually, by assumption \eqref{i} it follows that $u=0$ in
$\Omega\times (0,+\infty)$.
 \end{proof} 

We will need the following technical result from the proof of
\cite[Theorem 3.3]{mn1}.
\begin{Alemma}\label{lemmaest}Let $\eps>0$ and $c_\eps^\dagger$ be as
  in Theorem \ref{maximalspeedrd}. Then for every $c >0$ and every
  $u\in H^1_c(\Sigma)$ we have
  \begin{equation}\label{cvar} \Phi^\eps_c (u)\geq
    \frac{c^2-(c_\eps^\dagger)^2}{c^2 }\int_{\Sigma} e^{c z}
    \frac{\eps}{2} |u_z|^2 dx . 
  \end{equation} 
\end{Alemma}

\begin{proof}
  We define 
  \begin{equation}\label{utilda} 
    \tilde{u}(y,z) :=u\left(y,
      \frac{c^\dagger_\eps}{c} z\right).
  \end{equation} 
  Note that $ \tilde{u}\in H^1_{c_\eps^\dagger}(\Sigma)$. By a simple
  change of variables we then get
\begin{eqnarray*}
  \Phi^\eps_c (u) &=& \frac{c_\eps^\dagger}{c} \int_\Sigma
  e^{c_\eps^\dagger z} \left[\frac{\eps}{2} |\nabla \tilde{u}|^2 +
    \frac{\eps}{2} \left(\frac{c}{c_\eps^\dagger}\right)^2 |
    \tilde{u}_z|^2 +\frac{W(\tilde{u})}{\eps}-G(y,\tilde{u})\right]dx 
  \\
  &=&  \frac{c^\dagger_\eps}{c} \Phi^\eps_{
    c_\eps^\dagger}(\tilde{u})+  \frac{c^2-(c_\eps^\dagger)^2}{c
    c_\eps^\dagger}\int_{\Sigma} e^{c_\eps^\dagger z}  \frac{\eps}{2}
  |\tilde{u}_z|^2 dx,
\end{eqnarray*} 
which gives the result, since $\Phi^\eps_{ c_\eps^\dagger}(\tilde{u})\geq 0$. 
 \end{proof} 

We now state our main result.

\begin{Atheorem}\label{gammaconvergence}
  Let Assumptions \ref{g}, \ref{f} and \ref{h4} hold.  Let
  $c_\eps^\dagger$, $\bar u_\eps$ and $v_\eps$ be as in Theorem
  \ref{maximalspeedrd}, and let $c^\dagger$ be as in Theorem
  \ref{maximalspeedfmc}.
  \begin{itemize}\item[i)] There holds
  \begin{equation} \label{convc} \lim_{\eps\to 0} c_\eps^\dagger =
    c^\dagger.
  \end{equation}
 
\item[ii)] For every sequence $\eps_n\to 0 $ there exist a subsequence
  (not relabeled) and an open set $S \subset \Sigma$ such that
  \[
  \bar u_{\eps_n} \to \chi_S \qquad \text{ in }L^1_{loc}(\Sigma), 
  \] 
  where $S$ is a non-trivial minimizer of $\mathcal F_{c^\dagger}$
  satisfying $S\subseteq \Omega\times(-\infty,0)$ and $\partial S \cap
  (\overline\Omega \times \{0\}) \not= \emptyset$.  Moreover,
  \[
  \bar u_{\eps_n} \to \chi_S \qquad \text{ locally uniformly on
  }\overline\Sigma\setminus\partial S,
  \] 
  and for every $\theta\in (0,1)$ the level sets $\{\bar
  u_{\eps_n}=\theta\}$ converge to $\partial S$ locally uniformly in
  the Hausdorff sense.

\item[iii)] If also Assumption \ref{h6} holds, then $S$ is the unique
  minimizer of $\mathcal F_{c^\dagger}$ from Theorem
  \ref{maxspeedunique} satisfying $S\subseteq \Omega\times(-\infty,0)$
  and $\partial S \cap (\overline\Omega \times \{0\}) \not=
  \emptyset$. Moreover
  \[
  v_{\eps}\to 1 \qquad \text{ uniformly in }\overline\Omega. 
  \]
  \end{itemize}
\end{Atheorem}

\begin{proof} 
  We divide the proof into four steps.

  \medskip

  \noindent {\bf Step 1:} we shall prove that $$\liminf_{\eps\to
    0}c_\eps^\dagger\geq c^\dagger.$$ The proof follows by the
  standard Modica-Mortola construction of a recovery sequence
  \cite{momo}. Let $S_\psi$ be as in Theorem
  \ref{maximalspeedfmc}. Then the hypersurface ${\partial
    S_\psi\cap\Sigma}$ is of class $C^2$ on compact subsets of
  $\Sigma$ and of class $C^1$ on compact subsets of $\overline
  \Sigma$. By Proposition \ref{pminFc}(iii) the set $S_\psi$ satisfies
  \begin{equation}\label{minimal}
    c_W\Per_{c^\dagger}(S_\psi, \Sigma)=  \int_{S_\psi} e^{c^\dagger z} g(y)dx. 
  \end{equation} 

  Let now $d_{S_\psi}$ be the signed distance function from $\partial
  S_\psi$, i.e.,
  \[
  d_{S_\psi}(x) := {\rm dist}(x, \Sigma\setminus S_\psi)-{\rm dist}(x,
  S_\psi)
  \] 
  and $\gamma:\R\to \R$ to be the unique solution to $\gamma'=\sqrt{2
    W(\gamma)}$ with $\gamma(0)=\frac{1}{2}$.  Note that the map
  $t\mapsto \gamma(t)$ is monotone increasing and, by Assumption
  \ref{f}, converges exponentially to $0$ for $t\to -\infty$ and to
  $1$ for $t\to +\infty$. For $\eps,M>0$, we let
  \[
  u_{\eps,M}:=
  \gamma\left(\frac{d_{S_{\psi}}}{\eps}\right)\eta\left(\frac{z+M}{\eps}\right),
  \]  
  where $\eta:\R\to [0,1]$ is a smooth increasing function such that
  $\eta(z)=1$ for all $z\ge 1$ and $\eta(z)=0$ for all $z\le 0$.
  Since $\psi$ is bounded from above by Theorem \ref{maximalspeedfmc},
  we have that $u_{\eps,M}\in H^1_{c^\dagger}(\Sigma)$ for all $M$ and
  $\eps$ small enough.  Moreover, for $M>\sup_\Omega \psi$ we have
  $u_{\eps,M}\to \chi_{S_\psi\cap \Sigma_M}$ in $L^1(\Sigma)$ as
  $\eps\to 0$.  For $M>\sup_\Omega \psi$, we compute
  \begin{eqnarray*}
    \Phi^\eps_{c^\dagger}(u_{\eps,M}) &\le&  \int_{\Sigma} e^{c^\dagger z}
    \sqrt{2 W(u_{\eps,M})}|\nabla u_{\eps,M}| dx 
		\\
		&&-\int_{\Sigma} e^{c^\dagger
      z}G(y,u_{\eps,M})dx + Ce^{-c^\dagger M}
    \\
    &=&  \int_\Sigma e^{c^\dagger z} |\nabla \phi(u_{\eps,M})| dx
    -\int_\Sigma e^{c^\dagger z}G(y,u_{\eps,M})dx + Ce^{-c^\dagger M},
  \end{eqnarray*}
  where $\phi$ is as in \eqref{phi}, and the constant $C>0$ is
  independent of $\eps$ and $M$.  Notice that, by the $C^2$-regularity
  of ${\partial S_\psi\cap\Sigma}$, we have
	$$\Per_{c^\dagger}(\{\phi(u_{\eps,M})>t\},\Sigma)
	\to \Per_{c^\dagger}(S_\psi\cap \Sigma_M,\Sigma),$$ for any
        $t\in (0,c_W)$, as $\eps\to 0$.  Recalling the definition of
        $g$ in \eqref{gdef2}, as $\eps \to 0$ we also have that
        $\phi(u_{\eps,M})\to c_W\chi_{S_\psi\cap \Sigma_M}$ in
        $L^1(\Sigma)$. Therefore, we can apply the co-area formula
        (see \cite{afp}) and, possibly increasing the value of $C$, we
        obtain that
	%, as $\eps \to 0$,
  \begin{eqnarray}\nonumber
    \lim_{\eps\to 0}\Phi^\eps_{c^\dagger}(u_{\eps,M})&\le&
    \lim_{\eps\to 0}\Big(\int_0^{c_W}
    \Per_{c^\dagger}(\{\phi(u_{\eps,M})>t\},\Sigma)dt 
    \\\label{conv1}
    &&-\int_\Sigma
    e^{c^\dagger z}G(y,u_{\eps,M})dx  + Ce^{-c^\dagger M}\Big)
    \\  \nonumber
    &=& 
    c_W \Per_{c^\dagger}(S_\psi\cap \Sigma_M,\Sigma)-\int_{S_\psi}
    e^{c^\dagger z}g(y)dx + Ce^{-c^\dagger M} 
    \\\nonumber
    &\le& 
    c_W \Per_{c^\dagger}(S_\psi,\Sigma)-\int_{S_\psi}
    e^{c^\dagger z}g(y)dx + Ce^{-c^\dagger M} 
    \\\nonumber
    &=& Ce^{-c^\dagger M},
  \end{eqnarray} 
where the last equality follows from \eqref{minimal}.

  Assume now by contradiction that there exists a sequence of
  $c_{\eps}^\dagger$ converging to a constant $c<c^\dagger$. By
  \eqref{cvar} we have
  \begin{equation}\label{cvar2}
    \Phi_{c^\dagger}^\eps(u_{\eps,M}) \geq
    \frac{(c^\dagger)^2-(c_\eps^\dagger)^2}{(c^\dagger)^2 }\int_{\Sigma}
    e^{c^\dagger z} \frac{\eps}{2} \left|(u_{\eps,M})_z\right|^2 dx ,
  \end{equation}
  and observe that by the definition of $u_{\eps,M}$ and the
  regularity of $\partial S_\psi$, for $M$ large enough and $\eps$
  small enough independently of $M$, we get that
  \[
  |\nabla \phi(u_{\eps,M})| = \eps |\nabla u_{\eps,M}|^2 \leq 2 \eps
  |(u_{\eps,M})_z|^2,
  \] 
  in a ball $B(x,r)$ for some $r > 0$, where $x = (y,z) \in
  \partial S_\psi$ and $y \in \overline\Omega$ is a point at
  which $\psi$ attains its maximum. Combining these two facts yields
  \begin{eqnarray*}
    \Phi_{c^\dagger}^\eps(u_\eps) &\geq& 
    \frac{(c^\dagger)^2-(c_\eps^\dagger)^2}{4 (c^\dagger)^2 
    }\int_{\Sigma\cap B(x,r)} e^{c^\dagger  z}  
    |\nabla \phi(u_\eps)| dx 
    \\
    &\to& \frac{ (c^\dagger)^2-c^2}{4
      (c^\dagger)^2  } c_W \Per_{c^\dagger}(S_\psi, \Sigma\cap B(x,r))=:L>0,
  \end{eqnarray*}
  as $\eps\to 0$.  This contradicts \eqref{conv1}, by taking $M$ such
  that $Ce^{-c^\dagger M} <L$.

  \medskip

  \noindent {\bf Step 2:} let us now prove (i).  By Proposition
  \ref{pmaxspeede}, $c_\eps^\dagger$ is bounded from above by a
  constant independent of $\eps$.  In particular, there exists $c\in
  [0,+\infty)$ such that $c_\eps^\dagger\to c$ as $\eps\to 0$,
  along a sequence.  By Step 1 we have $c\ge c^\dagger$, so that it is
  enough to prove that $c\le c^\dagger$ for every sequence $\eps \to 0$.
  
  Recall that, for $\eps$ sufficiently small, we have $0 \leq \bar
  u_\eps \leq 2$ and $\ou_\eps(y,z)\leq \frac{1}{2}$ for every
  $y\in\Omega$ and $z>0$. By \eqref{pot} and Theorem
  \ref{maximalspeedrd}(iv), this implies that for $M>0$ and $\eps$
  sufficiently small we have
  \begin{eqnarray} \label{phi1} 0 & = &
    \Phi^\eps_{c_\eps^\dagger}(\ou_\eps) \nonumber \\ & \geq &
    \int_{\Sigma_M}e^{c_\eps^\dagger z} \left( \frac{\eps}{2} |\nabla
      \bar u_\eps|^2 + \frac{W(\bar u_\eps)}{\eps} - G(y, \bar u_\eps)
    \right) dx - \frac{|\Omega|}{c^\dag_\eps} \| G
    \|_{2,\infty} e^{-c^\dag_\eps M}  \nonumber \\
    & \geq & \int_{\Sigma_M}e^{c_\eps^\dagger z}
    \left(\sqrt{2W(\ou_\eps)}|\nabla \ou_\eps |-G(y,\ou_{\eps})
    \right)dx - \frac{2 |\Omega|}{c} \| G \|_{2,\infty} e^{-c M}
    \nonumber \\
    & = & \int_{\Sigma_M}e^{c_\eps^\dagger z} \left( |\nabla
      \phi(\ou_\eps) | -G(y,\ou_{\eps}) \right) dx - \frac{2 |\Omega|}{
      c} \| G \|_{2,\infty} e^{-c M},
  \end{eqnarray}
  where $\phi$ is as in \eqref{phi}.  By Lemma \ref{coer} we get, up
  to a subsequence, that
  \begin{equation}\label{eqstep}\ou_\eps\to\chi_S\qquad \text{in }
    L^1_{loc}(\Sigma),
  \end{equation}
  where $\chi_S \in BV_{loc}(\Sigma)$ and
  $S\subseteq\Omega\times(-\infty,0)$. Moreover, by \eqref{lead} and
  the density estimate in Proposition \ref{densityrd} we have $\int_S
  e^{cz} dx > 0$.

  By the lower semicontinuity in $BV$ (see \cite{afp}) of the
  functional
  \[
  u\mapsto \int_{\Sigma_M}e^{c z} \left( |\nabla \phi(u) | -G(y,u)
  \right) dx,
  \] 
  and by the fact that $\phi(u_\eps)\to \phi(\chi_S)=c_W \chi_S$ in
  $L^1(\Sigma_M)$, we get
  \begin{multline}
    \liminf_{\eps \to 0} \int_{\Sigma_M}e^{c_\eps^\dagger
      z}\left(|\nabla \phi(\ou_\eps)| -G(y,\ou_{\eps})\right)dx \\
    \geq c_W \Per_c(S, \Sigma_M )-\int_{S \cap \Sigma_M} e^{cz} g(y)
    dx. \qquad
    \label{phi2}
  \end{multline}
  Sending now $M \to \infty$, from \eqref{phi1} and \eqref{phi2} we
  conclude that
  \begin{equation}\label{mini}
    \mathcal{F}_c(S)\leq 0,
  \end{equation} 
  which, by Proposition \ref{pminFc}(ii), implies that $c\leq
  c^\dagger$.

  \medskip

  \noindent {\bf Step 3:} we now prove (ii).  By \eqref{convc} it
  follows that \eqref{mini} holds with $c=c^\dagger$. Therefore, by
  Proposition \ref{pminFc}(iii) the inequality in \eqref{mini} is in
  fact an equality, and by Remark \ref{remden} and the density
  estimate \eqref{dens} the set $S$ is a non-trivial minimizer of
  $\mathcal{F}_{c^\dagger}$, satisfying all the desired
  properties. Furthermore, $S$ is the subgraph of a function $\psi :
  \Omega \to [-\infty, \infty)$ that satisfies all the conclusions of
  Theorem \ref{maximalspeedfmc}.

  Let $\theta\in (0,1)$ and assume by contradiction that the level
  sets $\{\bar u_{\eps}=\theta\}$ do not converge to $\partial S$
  locally uniformly in the Hausdorff distance. This means that there
  exist $\delta,M>0$ and points $x_\eps\in\Sigma_M$ such that $\bar
  u_{\eps}(x_\eps)=\theta$ and ${\rm dist}(x_\eps,\partial S)\geq
  \delta>0$.  Up to extracting a subsequence we can assume that
  $x_\eps\in S$ or $x_\eps\in \Sigma\setminus S$, for all
  $\eps$. Assume $x_\eps\in S$, and let $x\in S$, with ${\rm
    dist}(x,\partial S)\geq \delta$, such that $x_\eps\to x\in S$ as
  $\eps\to 0$.  By \eqref{eqstep} we have that $\ou_\eps\to 1$ in
  $L^1(B(x,\delta/2)\cap\Sigma)$, which contradicts the density
  estimate \eqref{densbis}.  If
  $x_\eps\in\Sigma\setminus S$ one can reason analogously,
  contradicting the density estimate \eqref{dens}.

  The locally uniform convergence of $\ou_\eps$ to $\chi_S$
  outside $\partial S$ is a direct consequence of the convergence of
  the level sets $\{\bar u_{\eps}=\theta\}$ to $\partial S$ in the
  Hausdorff sense.

  \medskip

  \noindent {\bf Step 4:} it remains to prove (iii). By Theorem
  \ref{maxspeedunique} there exists a unique minimizer $S$ of
  $\mathcal{F}_{c^\dagger}$ such that $S\subseteq
  \Omega\times(-\infty,0)$ and $\partial S \cap (\overline\Omega
  \times \{0\}) \not= \emptyset$, and so $\ou_\eps \to \chi_S$ in
  $L^1(\Sigma_M)$ for every $M>0$.  Moreover, in this case $S$ is the
  subgraph of a bounded function $\psi$ defined in $\Omega$, so we get
  that $\chi_S(y,z)\equiv 1$ for all $y\in \Omega$ and $z<
  \min_{\overline{\Omega}}\psi$. Then from the locally uniform
  convergence proved in Step 3, the monotonicity of $\ou_\eps(y,z)$ in
  $z$ and the fact that $\ou_\eps \leq 1 + C \eps$ for some $C > 0$
  and $\eps$ small enough (see Theorem \ref{maximalspeedrd}), we have
  $\ou_\eps(y,z)\to 1$ uniformly in $\Omega\times (-\infty,M]$ for
  every $M< \min_{\overline{\Omega}}\psi$. The conclusion then follows
  from the fact that again by Theorem \ref{maximalspeedrd} we have
  $\ou_\eps(y,z)\leq v_\eps(y)\leq 1+C\eps$ for every $(y,z)\in
  \Sigma$.
 \end{proof} 

The result in Theorem \ref{gammaconvergence} allows us to make an
important conclusion about spreading of the level sets of solutions of
the initial value problem with general front-like initial data for
$\eps \ll 1$. We define the {\em leading edge}, i.e., the
quantity\footnote{The definition in \eqref{Rdel} corrects a typo in
  \cite[Eq. (5.1)]{mn1}.}
  \begin{align}
    \label{Rdel}
    R_\theta^\eps(t) := \sup \{z \in \R : \ u^\eps(y, z, t) > \theta \
    \text{for some} \ y \in \Omega \},
  \end{align}
  with $\theta > 0$, for the solution $u^\eps$ of \eqref{rd}. Then the
  following result is an immediate consequence of Theorem
  \ref{gammaconvergence} and \cite[Theorem 5.8]{mn1}.
 
  \begin{Acorollary} 
    \label{r:lead} 
    Let $u^\eps$ be a solution of \eqref{rd} with initial datum
    $u^\eps_0 \in W^{1,\infty}(\Sigma) \cap L^2_c(\Sigma)$ for some $c
    > c^\dag$, where $c^\dag$ is as in Theorem
    \ref{maximalspeedfmc}. Assume that $u_0^\eps \leq 1 + \delta$ in
    $\Sigma$, where $\delta$ is as in Remark \ref{max}, and
    $u^\eps_0(\cdot, z) \geq 1 + C \eps$ for all $z\le M$, for some $M
    \in \mathbb R$, where $C$ is as in \eqref{pit}. Then under
    Assumptions \ref{g}, \ref{f}, \ref{h4} we have
  \begin{align}
    \label{Rdelinf}
    \lim_{\eps \to 0} \lim_{t \to \infty} \frac{R^\eps_\theta(t)}{t} =
    c^\dag,
  \end{align}
  for all $\theta \in (0,1)$, where $R_\theta^\eps(t)$ is given by
  \eqref{Rdel}.
\end{Acorollary}

Thus, $R_\theta^\eps(t)$ propagates, for $\eps$ small enough,
asymptotically as $t \to \infty$ with the average speed that
approaches $c^\dag$ as $\eps \to 0$. The fact that $\theta$ can be
chosen arbitrarily from $(0,1)$ follows by inspection of the proof of
\cite[Theorem 5.8]{mn1} and the conclusion of Theorem
\ref{gammaconvergence}(ii).

We now investigate the long-time behavior of the solutions of
\eqref{rd} in more detail. Under Assumption \ref{h6}, which is
stronger than our standing Assumption \ref{h4}, we show that the
long-time limit of solutions to \eqref{rd} with front-like initial
data converges, as $\eps\to 0$, to a traveling wave solution to
\eqref{fmc} moving with speed $c^\dagger$.

\begin{Atheorem}\label{stability} 
  Let Assumptions \ref{g}, \ref{f} and \ref{h6} hold.  Let $\delta>0$
  be such that 
  $$(1-u) f(u) > 0 \qquad \text{for all}\qquad u\in [1-\delta,1) \cup
  (1,1+\delta],$$ let $u_{0}^\eps\in W^{1,\infty}(\Sigma)\cap
  L^2_{c^\dagger_\eps}(\Sigma)$ be such that
  \begin{align}
    \label{u0conv}
    0\leq u_0^\eps\leq 1+\delta \qquad \text{and} \qquad \liminf_{z\to
      -\infty} u_{0 }^\eps(y,z) \ge 1 - \delta \ \text{uniformly
      in $\Omega$},
  \end{align}
  and let $u^\eps$ be the solution of \eqref{rd} with initial datum
  $u_0^\eps$. Then there exists $R_\infty \in \mathbb R$ such
    that, for all $M>0$,
  \begin{equation}\label{limlim}
    \lim_{\eps\to 0}\ \lim_{t\to \infty} \|u^\eps(y,z + c_\eps^\dagger t
    + R_\infty, t)-\chi_{S_\psi}(y,z) \|_{L^1(\Sigma_M)} =0\,,
  \end{equation}
  where $\psi$ is given by Theorem \ref{maxspeedunique}. Moreover, the
  convergence as $\eps \to 0$ after passing to the limit $t \to
  \infty$ is locally uniform in $\overline \Sigma \backslash \partial
  S_\psi$.
\end{Atheorem} 

\begin{proof} 
  The proof follows from Theorem \ref{gammaconvergence} and from the
  stability results in Theorem 1 and Corollary 2.1 of \cite{mn3},
  which apply to the solutions of the initial value problem for
  \eqref{rd} under the additional assumption that
  $v_\eps(y)=\lim_{z\to -\infty} \ou_\eps(y,z)$ is a nondegenerate
  stable critical point of $E^\eps$. We note that under our
  assumptions the value of $\alpha$ in Theorems 1 and 3 of \cite{mn3}
  does not depend on the parameter $\eps$. This is due to the fact
  that $u=1 - \delta$ is a subsolution for \eqref{rd} for all
 $\eps$ small enough. Thus, to conclude
  we only need to demonstrate that under the assumptions of the
  theorem $v_\eps$ is indeed non-degenerate. This is proved in Lemma
  \ref{unique} below.
 \end{proof} 

\begin{Alemma}\label{unique} 
  Let Assumptions \ref{g}, \ref{f} and \ref{h6} hold and let
  $(c_\eps^\dagger, \ou_\eps)$ be as in Theorem
  \ref{maximalspeedrd}.  Then there exists $\eps_0>0$ such that for
  every $0<\eps<\eps_0$, $v_\eps(y)=\lim_{z\to -\infty} \ou_\eps(y,z)
  $ is a nondegenerate stable critical point of $E^\eps$.
\end{Alemma} 
\begin{proof}
  By Theorem \ref{gammaconvergence}(iii), we have that $v_\eps\to 1$ 
  uniformly in $\overline\Omega$. 
  Fix $\delta>0$ such that $W''(u)>0$ for every $u\in
  [1-\delta,1+\delta]$.  Let $\eps_0$ be such that for all
  $\eps<\eps_0$ we have $v_\eps(y)\in (1-\delta, 1+\delta)$ for all
  $y\in \overline{\Omega}$. Moreover, eventually decreasing $\eps_0$,
  we have that \[\frac{W''(s) }{\eps} - G_{uu} (y,s)> 0 \qquad\forall
  y\in\overline{\Omega}, \ s\in [1-\delta, 1+\delta], \ \eps<\eps_0.\]
  This implies that $v_\eps$ is a non degenerate stable critical point
  of $E_\eps$.
 \end{proof}

\begin{Aremark}\rm
To derive the stability result in Theorem \ref{stability}, it is
  essential that the local minimizer $v_\eps$ of $E^\eps$ to which the
  traveling wave $(c_\eps^\dagger, \ou_\eps)$ is converging as $z\to
  -\infty$ is non-degenerate, according to Definition \ref{locmin}.
In general the assumption that $v_\eps$ is nondegenerate is quite
  difficult to check, even if it is generically satisfied, see the
  discussion in \cite{mn3}.  In Lemma \ref{unique}, we show that a
  sufficient condition for it is Assumption \ref{h6}, together with
  Assumptions \ref{g} and \ref{f}. More generally,
we expect that the same nondegeneracy condition on
$v_\eps$ is generically true, when there is at most one set
$\omega \subseteq \Omega$ such that $\omega \times \R$ is a
minimizer under compact perturbations of the geometric functional
$\mathcal{F}_{c^\dagger}$ and, moreover, this unique local minimizer
has positive second variation.
\end{Aremark}

Lastly, we briefly discuss what kinds of counterparts to our
propagation results can be obtained, using the methods of 
\cite{bss,barles}. We note that because of the local in time nature of
convergence in \cite{bss,barles}, the order of the limits in
\eqref{limlim} in such results needs to be reversed. Then the
conclusion can be obtained via the analysis of the long time limit of
\eqref{fmc}, as is done, e.g., in \cite{cn}. To be specific, if the
initial data $u^\eps_0$ converge to $\chi_{S_{h_0}}$ locally uniformly
out of $\partial S_{h_0}$, for some $h_0\in W^{1,\infty}(\Omega)$,
then by \cite{bss,barles} the solutions $u^\eps$ of \eqref{rd} with
initial data $u^\eps_0$ converge locally uniformly to $\chi_{S_{h}}$,
where $h$ is the solution of \eqref{fmc} with initial datum
$h_0$. Since by \cite{cn}, under Assumption \ref{h6}, the function
$h(y,t)- c^\dagger t -R_\infty$ converges uniformly to $\psi(y)$ as
$t\to +\infty$, it follows that
\begin{equation}\label{lamlam}
  \lim_{t\to \infty} \ \lim_{\eps\to 0}\, \|u^\eps(y,z + c^\dagger t
  + R_\infty, t)-\chi_{S_\psi}(y,z) \|_{C(K)} =0\,, 
\end{equation}
for any compact set $K\subset\overline\Sigma\setminus \partial
S_\psi$. Thus, the expectation about the long time behavior of
solutions of \eqref{rd} for $\eps \ll 1$ based on the analysis of the
mean curvature flow that follows from \eqref{lamlam} is justified by
our result in Theorem \ref{stability}.

\medskip

\paragraph*{Acknowledgements.}
The work of {\sc Annalisa Cesaroni} and {\sc Matteo Novaga} was
partially supported by the Fondazione CaRiPaRo Project ``Nonlinear
Partial Differential Equations: models, analysis, and
control-theoretic problems.'' The work of {\sc Cyrill B. Muratov} was
supported, in part, by NSF via grants DMS-0908279, DMS-1119724 and
DMS-1313687.

\appendix
  
\section*{Appendix}

\section{Density estimates}
\label{sec:appendix}

In this Appendix we establish a general density estimate in the spirit
of \cite{cc,fv,nv} for minimizers of Allen-Cahn type
functionals. Note, however, that our estimates are in terms of the
averages of the $L^2$ norms of the minimizers with respect to
compactly supported perturbations, rather than in terms of the
densities associated with their superlevel sets. The key ingredient of
the proof is still an application of the Gagliardo-Nirenberg-Sobolev
inequality, as in \cite{cc,fv,nv}. However, the use of a simpler test
function and of $L^2$ estimates makes the proof considerably more
straightforward. In fact, our proof is in some sense more along the
lines of the respective density estimates for minimal surfaces and
relies in an essential way on the Modica-Mortola trick
\cite{momo}. Also, we point out that our functionals, as in \cite{nv}
and in contrast to \cite{cc,fv}, do not necessarily admit minimizers
that are constants. Our assumptions are more general than those of
\cite{nv}, however, since they do not require $G(\cdot, u)$ to have
zero mean.

\begin{Atheorem}
  \label{p:densL2}
  For $\rho > 2$ and $u \in H^1(B(0,\rho)) \cap
  L^\infty(B(0,\rho))$, let
  \begin{align}
    \label{HH}
    H(u) := \int_{B(0, \rho)} \Big(( a(x) \nabla u) \cdot \nabla u +
    b(x) W(u) + G(x, u) \Big) dx,
  \end{align}
  where $W$ is defined by \eqref{pote} with $f$ satisfying Assumption
  \ref{f}, $a(x)$ is a symmetric $n \times n$ matrix, $a \in
  W^{1,\infty}(B(0, \rho); \R^{n \times n})$, $b \in L^\infty(B(0,
  \rho))$, $G(x, u)$ is a Carath\'eodory function, and $a$ and $b$
  satisfy
  \begin{align}
    \label{Hlam}
    \lambda \leq b(x) \leq \lambda^{-1} \quad \text{and} \quad \
    \lambda |\xi|^2 \leq (a(x) \xi) \cdot \xi \leq \lambda^{-1}
    |\xi|^2 \qquad \ \forall x \in B(0, \rho), \ \forall \xi \in \R^n,
  \end{align}
  for some $\lambda > 0$. Then there exists $r_0 \in \mathbb N$
  depending only on $n$, $W$, $\| a \|_{W^{1,\infty}(B(0, \rho); \R^{n
      \times n})}$ and $\lambda$ such that if $u$ is a minimizer of
  $H$ with prescribed boundary data on $\partial B(0, \rho)$, $\|u -
  \frac12 \|_{L^\infty(B(0, \rho))} \leq 1$, $\alpha \in (0,
  r_0^{1-n})$, $R_0$ is an integer such that $r_0 + 1\leq R_0 < \rho$,
  and $\| G \|_{L^\infty(B(0, \rho)\times (-\frac12,\frac32))} \leq
  \alpha R_0^{-1}$, then
  \begin{align}
    \label{eq:densL2}
    \dashint_{B(0, r_0)} u^2 dx \geq \alpha \qquad &\Rightarrow \qquad
    \dashint_{B(0, R)} u^2 dx \geq \alpha,
    \\
    \label{eq:densL3}
    \dashint_{B(0, r_0)} (1-u)^2 dx \geq \alpha \qquad &\Rightarrow
    \qquad \dashint_{B(0, R)} (1-u)^2 dx \geq \alpha,
  \end{align}
  for all $R\in [r_0, R_0]$ integer.
\end{Atheorem}

\begin{proof}
  We only prove \eqref{eq:densL2}, since \eqref{eq:densL3} then
  follows by a change of variable $u \to1-u$.  Let $\theta \in
  C^\infty(\R)$ with $\theta(x) = 0$ for all $x < 0$, $\theta(x) = 1$
  for all $x > 1$ and $\theta'(x) \geq 0$ for all $x \in \R$. For $1
  \le R < \rho-1$, let $\eta(x) := \theta(|x|-R)$ be a cutoff
  function. Since $u$ is a minimizer of $H$ with respect to
  perturbations supported in $B(0, \rho)$, we have $H(u) \leq H(u
  \eta)$. Then by positivity of $a$, $b$ and $W$ and the fact that
  $\eta = 0$ in $B(0,R)$, we obtain
  \begin{multline}
    \label{Htest}
    \int_{B(0, R+1)} (1 - \eta^2) \Big( (a(x) \nabla u) \cdot \nabla u
    + b(x) W(u) \Big) dx \leq
    2 |B(0,R+1)| G_0 \\
    + \int_{B(0, R + 1) \backslash B(0,R)} \left( \frac12 (a(x) \nabla
      u^2) \cdot \nabla \eta^2 + u^2 (a(x) \nabla \eta) \cdot \nabla
      \eta + b(x) W(u \eta) \right) dx,
  \end{multline}
  where we introduced
  \begin{align}
    \label{HG0}
    G_0 := \| G \|_{L^\infty(B(0, \rho) \times (-\frac12, \frac32))}.
  \end{align}
  Integrating by parts the first term in the integral on the
  right-hand side of \eqref{Htest} and using the assumptions on $a$
  and $b$ in the left-hand side, we have
  \begin{multline}
    \label{Htest2}
    \lambda \int_{B(0, R+1)} (1 - \eta^2) \Big( |\nabla u|^2 + W(u)
    \Big) dx \leq
    2 |B(0,R+1)| G_0 \\
    + \int_{B(0, R + 1) \backslash B(0,R)} \Big( b(x) W(u \eta) - u^2
    \eta\nabla \cdot (a \nabla \eta) \Big) dx. \qquad
  \end{multline}
  Therefore, since our assumptions imply that $W(s) \leq C_1 s^2$ for
  all $|s| \leq \frac32$ and some $C_1 > 0$ depending only on $W$, we
  have
 \begin{multline}
    \label{Htest3}
    \int_{B(0, R+1)} (1 - \eta^2) \Big( |\nabla u|^2 + W(u) \Big) dx \\
    \leq C \left( R^n G_0 + \int_{B(0, R + 1) \backslash B(0,R)} u^2
      dx \right),
  \end{multline}
  for some constant $C > 0$ depending only on $n$, $W$, $\| a
  \|_{W^{1,\infty}(\Om; \R^{n \times n})}$ and $\lambda$, which
  changes from line to line from now on. In particular, since by our
  assumptions $W(s) \geq C_2 s^2$ for all $|s| \leq \frac12$ and some
  $C_2 > 0$ depending only on $W$, we obtain
  \begin{align}
    \label{HWL2}
    \int_{B(0, R) \cap \{ |u| \leq \frac12 \} } u^2 dx \leq C \left(
      R^n G_0 + \int_{B(0, R + 1) \backslash B(0,R)} u^2 dx \right).
  \end{align}

  We now use the Modica-Mortola trick \cite{momo} and estimate the
  left-hand side of \eqref{Htest3} from below as follows:
  \begin{align}
    \label{Hmomo}
    \int_{B(0, R+1)} |(1 - \eta^2) \nabla \phi(u)| dx \leq C \left(
      R^n G_0 + \int_{B(0, R + 1) \backslash B(0,R)} u^2 dx \right),
  \end{align}
  where $\phi(u)$ is defined via \eqref{phi}. Therefore, with the help
  of Gagliardo-Nirenberg-Sobolev inequality we get
  \begin{multline}
    \label{HGNS}
    \left( \int_{B(0, R+1)} |(1 - \eta^2) \phi(u)|^{\frac{n}{ n-1}} dx
    \right)^{\frac{n-1}{ n}} \\ \leq C \left(R^n G_0 + \int_{B(0, R + 1)
        \backslash B(0,R)} \left( u^2 + \phi(u) |\nabla \eta^2|
      \right) dx \right).
  \end{multline}
  Moreover, since by our assumptions $C_3 s^2 \leq |\phi(s)| \leq C_4
  s^2$ for all $|s| \leq \frac32$ and some $C_3, C_4 > 0$ depending
  only on $W$, we have
  \begin{align}
    \label{HSob}
    \left( \int_{B(0, R)} |u|^{2\frac{n}{ n-1}} dx
    \right)^{\frac{n-1}{ n}} \leq C \left( R^n G_0 + \int_{B(0, R + 1)
        \backslash B(0,R)} u^2 dx \right).
  \end{align}
  Raising both sides of this inequality to the power $n/(n-1)$,
  we obtain
  \begin{align}
    \label{HSob2}
    \int_{B(0, R) \cap \{ |u| > \frac12 \} } u^2 dx \leq C \left( R^n
      G_0 + \int_{B(0, R + 1) \backslash B(0,R)} u^2 dx \right)^{\frac{n}{ n-1}} .
  \end{align}

  Let us introduce the quantity
 \begin{align}
    \label{HM}
    M_R := \int_{B(0,R)} u^2 dx + |B(0,1)| R^{n+1} G_0.
  \end{align}
  Adding \eqref{HWL2} and \eqref{HSob2}, and expressing the
  result in terms of $M_R$ yields
  \begin{multline}
    \label{HL2}
    M_R - |B(0,1)| R^{n+1} G_0 \leq \\
    C \left( 1 + \left( R^n G_0 + \int_{B(0, R + 1) \backslash B(0,R)}
        u^2 dx \right)^{\frac{1}{n-1} }\right) (M_{R+1} - M_R).
  \end{multline}
  We can rewrite the inequality in \eqref{HL2} in the form
  \begin{align}
    \label{HL2L2}
    M_{R+1} \geq K(u, R) M_R,
  \end{align}
  where
  \begin{align}
    \label{HK}
    K(u, R) := 1 + \frac{C \left( 1 - 
	\frac{R G_0}{\dashint_{B(0,R)} u^2 dx
          + R G_0 } \right)}{ 1 + R \left( R G_0 + \dashint_{B(0,
          R + 1) \backslash B(0,R)} u^2 dx \right)^{\frac{1}{n-1}}
    }\geq 1.
  \end{align}

  Let now $r_0 \in \mathbb{N}$, and let $R_0 \in \mathbb N$ and
  $\alpha\in\R$ be such that $r_0+1 \leq R_0 < \rho$ and $\alpha\in
  (0,r_0^{1-n})$.  If $\dashint_{B(0, R + 1)} u^2 dx \geq \alpha$ for
  all $r_0 \leq R \leq R_0-1$ integer, then there is nothing to
  prove. So suppose the opposite inequality holds for some integer
  $r_0 \leq R_1 \leq R_0-1$, and that $R_1$ is the smallest value of
  $R$ for which this happens. Then
  \begin{align}
    \label{HaveRbar}
    \dashint_{B(0, R_1)} u^2 dx \geq \alpha \quad \text{and} \ \quad
    \dashint_{B(0, R_1+1)\backslash B(0,R_1)} u^2 dx < \alpha\,,
  \end{align}
  and we can estimate $K(u, R_1)$ from below as
  \begin{align}
    \label{HK2}
    K(u, R_1) & \geq 1 + \frac{C \alpha }{\left( R_0 G_0 + \alpha
      \right) \left( 1 + R_1 \left( R_0 G_0 + \alpha
        \right)^{\frac{1}{n-1}}  \right) } \notag \\
    & \geq 1 + \frac{C}{1 + R_1 \alpha^{\frac{1}{n - 1}} } \geq 1 + \frac{C
      r_0 }{2 R_1 }\,.
  \end{align}
  By \eqref{HM} and \eqref{HL2L2}, this implies that
  \begin{align}
    \label{Havs}
    \int_{B(0,R_1+1)} u^2 dx + 2^n (n+1) |B(0,1)| R_1^n G_0 \geq
    \left( 1 + \frac{C r_0}{R_1 } \right) \int_{B(0,R_1)} u^2 dx,
  \end{align}
  and, hence, by our assumptions and \eqref{HaveRbar} we obtain 
  \begin{align}
    \label{Hcontrad}
    \dashint_{B(0,R_1+1)} u^2 dx \geq \left( 1 + \frac{1}{R_1}
    \right)^{-n} \left( 1 + \frac{C r_0 - 2^n (n+1)}{R_1} \right)
    \alpha.
  \end{align}
  Since $R_1 \geq 1 $, choosing 
  \begin{align}
    \label{Hr0}
    r_0 = \left\lceil \frac{2^n (n+2)-1}{C} \right\rceil,
  \end{align}
  where $C$ is the constant appearing in \eqref{Hcontrad}, we get that
  the right-hand side of \eqref{Hcontrad} is greater or equal than
  $\alpha$, contradicting our assumption on $R_1$.
 \end{proof}

\begin{Aremark}\rm
  An inspection of the proof of Theorem \ref{p:densL2} shows that  
	$\dashint_{B(0,R)} u^2 dx$ is monotonically increasing in
  $R \in \mathbb N$, as long as it is not too big. More precisely,
  under the assumptions of Theorem \ref{p:densL2}, we have that
  $\dashint_{B(0,R)} u^2 dx$ is monotonically increasing for all $R
  \in [r_0, R_0]$ integer, provided that $\dashint_{B(0,R)} u^2 dx <
  r_0^{1-n}$ and $\| G \|_{L^\infty(B(0, \rho) \times (-\frac12,
    \frac32))} \leq r_0^{1-n} R_0^{-1}$.
\end{Aremark}

We also note that Theorem \ref{p:densL2} yields the kinds of density
estimates for the level sets of the minimizers of Ginzburg-Landau
functionals with respect to compactly supported perturbations that
were previously obtained in \cite{cc,fv,nv}. Here we give a result
that extends those of \cite{cc,fv,nv} to the case of the functional
$H$ in \eqref{HH}, generalizing the estimates obtained in \cite{cc,fv}
for the case of functionals that admit constant minimizers, and the
estimates of \cite{nv} under assumption of periodicity of $G$.

\begin{Acorollary}
  Under the assumptions of Theorem \ref{p:densL2}, let $\beta \in
  (0,1)$, and for $R > 0$ let
  \begin{align}
    \mu_{\beta,R} := |\{ |u| > \beta \} \cap B(0,R)|.
  \end{align}
  If $\mu_{\beta,1} > 0$, there exist $C, C' > 0$ depending on $n$,
  $W$, $\| a \|_{W^{1,\infty}(B(0, \rho); \R^{n\times n})}$,
  $\lambda$, $\beta$ and $\mu_{\beta,1}$, such that
  \begin{align}
    \label{HmuR}
    \mu_{\beta,R} \geq C R^n,
  \end{align}
  for all $R \in [1, \rho]$ satisfying $R \leq C' \| G
  \|^{-1}_{L^\infty(B(0, \rho) \times (-\frac12, \frac32))} $.
\end{Acorollary}

\begin{proof}
  Throughout the proof, $C,C'$ denote positive constants depending
  only on $n$, $W$, $\| a \|_{W^{1,\infty}(B(0, \rho); \R^{n \times
      n})}$, $\lambda$, $\beta$ and $\mu_{\beta,1}$ that may change
  from line to line.

  Let $r_0\geq 1$  be as in Theorem \ref{p:densL2}. Then 
  \begin{align}
    \label{Cu2r0}
    \dashint_{B(0,r_0)} u^2 dx \geq \frac{\beta^2 \mu_{\beta, 1} }{|B(0,
      r_0)|} =: \alpha \in (0,r_0^{1-n}). 
  \end{align}
  Also, clearly $\dashint_{B(0,R)} u^2 dx\geq \alpha$ for every $1\leq
  R\leq r_0$.  Therefore, by \eqref{Cu2r0} and Theorem \ref{p:densL2}
  we have that \eqref{eq:densL2} holds for any $R \in [r_0, R_0]$
  integer, provided that $r_0 + 1 \leq R_0 < \rho$ is an integer that
  satisfies $R_0 G_0 \leq \alpha$, where $G_0$ is defined in
  \eqref{HG0}. Extending this estimate to the whole interval then
  yields
  \begin{align}
    \label{Cu2R}
    \int_{B(0, R)} u^2 dx \geq 2^{-n} \alpha|B(0,R)|\qquad \qquad \forall
    R \in [1, R_0].
  \end{align}
  Moreover, since $\| u \|_{L^\infty(B(0,\rho))} \leq \frac32$ by the
  assumptions in Theorem \ref{p:densL2}, we can write
  \begin{align}
    \label{nR}
    \int_{B(0, R)} u^2 dx \leq \int_{B(0,R) \cap \{ |u| \leq \beta \}}
    u^2 dx + \frac{9}{4}\mu_{\beta, R}.
  \end{align} 
  On the other hand, by \eqref{HWL2} (with $\frac12$ replaced by
  $\beta$) we have
  \begin{align}
    \label{Cgamu2}
    \int_{B(0,R) \cap \{ |u| \leq \beta \}} u^2 dx \leq C ( R^{n-1}
    +R^n G_0 ) \qquad \forall R \in [1, \rho].
  \end{align}
  Therefore, by \eqref{Cu2R}, \eqref{nR} and \eqref{Cgamu2}, and
  recalling that $R_0 G_0 \leq\alpha < 1$, we obtain for all $R\in [1,
  R_0]$,
  \begin{align}
    \label{Cumu2}
    \frac{9}{4} \mu_{\beta,R}\geq 2^{-n} \alpha |B(0, 1)| R^n- C
    R^{n-1}.
  \end{align}
  From this we conclude that $\mu_{\beta,R} \geq C R^n$ whenever $R>
  C'$. At the same time, by definition $\mu_{\beta,R} \geq
  \mu_{\beta,1} \geq \mu_{\beta,1} (C')^{-n} R^n$ whenever $R \leq
  C'$. This concludes the proof.
 \end{proof}

% \bibliography{../nonlin}

\begin{thebibliography}{33}

\bibitem{almat} {\sc M.~Alfaro, D. Hilhorst, H.~Matano}: 
The singular limit of the Allen-Cahn equation and the FitzHugh-Nagumo
system. 
{\it J. Differential Equations} {\bf 245}, 505--565 (2008)

\bibitem{alfaro12} {\sc M.~Alfaro, H.~Matano}: 
On the validity
of formal asymptotic expansions in {A}llen-{C}ahn equation and
{F}itz{H}ugh-{N}agumo system with generic initial data.   
{\it Discrete Cont. Dyn. Syst. B} {\bf 17}, 1639--1649 (2012)

\bibitem{allen79} {\sc S. M. Allen, J. W. Cahn}: 
A microscopic theory for antiphase boundary motion and its
application to antiphase domain coarsening.
{\it Acta Metal. Mater.} {\bf 27}, 1085--1095 (1979)

\bibitem{amar08} {\sc M. Amar, V. De Cicco, N. Fusco}: 
Lower semicontinuity
and relaxation results in BV for integral functionals with BV
integrands. 
{\it ESAIM Control Opt.  Calc. Var.} {\bf 14}, 456--477, 2008

\bibitem{afp} {\sc L. Ambrosio, N. Fusco, D. Pallara}: 
{\it Functions of Bounded Variation and Free Discontinuity Problems.}
Oxford Mathematical Monographs, New York (2000)

\bibitem{aronson78} {\sc D.~G. Aronson, H.~F. Weinberger}: 
Multidimensional
diffusion arising in population genetics.  
{\it Adv. Math.} {\bf 30},
33--58 (1978)

\bibitem{bcn} {\sc G. Barles, A. Cesaroni, M. Novaga}: 
Homogenization of fronts in highly heterogeneous media. 
{\it SIAM J. Math. Anal.} {\bf 43}, 212--227 (2011)

\bibitem{bss} {\sc G. Barles, H. M. Soner, P. E. Souganidis}: 
Front propagation and phase field theory.
{\it SIAM J. Control Optim.} {\bf 31}, 439--469 (1993)

\bibitem{barles} {\sc G.~Barles, P.~E. Souganidis}: 
A new approach to front
propagation problems: Theory and applications. 
{\it Arch. Rational
Mech. Anal.} {\bf 141}, 237--296 (1998)

\bibitem{b2} {\sc H. Berestycki, L.  Nirenberg}: 
Travelling
fronts in cylinders. 
{\it Ann. Inst. H. Poincar\'e
Anal. Non Lin\'eaire} {\bf 9}, 497--572 (1992)

\bibitem{braides} {\sc A. Braides}: 
{\it $\Gamma$-convergence for
beginners}. Oxford Lecture Series in Mathematics and
its Applications, Oxford (2002)

\bibitem{bronsard91} {\sc L.~Bronsard, R.~V. Kohn}: 
Motion by
mean curvature as the singular limit of {G}inzburg-{L}andau
dynamics. 
{\it J. Differential Equations} {\bf 90}, 211--237 (1991)

\bibitem{cc} {\sc L. A.  Caffarelli, A. C\'ordoba}:
Uniform convergence of a singular perturbation problem.
{\it Comm. Pure Appl. Math.} {\bf 48}, 1--12 (1995) 

\bibitem{caginalp86ann} {\sc G.~Caginalp}: 
The role of
microscopic anisotropy in the macroscopic behavior of a phase
boundary.
{\it Ann. Phys.} {\bf 172}, 136--155 (1986)

\bibitem{caginalpfife} {\sc G.~Caginalp, P. C. Fife}:
Phase-field methods for interfacial boundaries.
{\it Phys. Rev. B} {\bf 33}, 7792--7794 (1986)

\bibitem{cls} {\sc P. Cardaliaguet, P.-L. Lions, P. E. Souganidis}: 
A discussion about the homogenization of moving interfaces.
{\it
J. Math. Pures Appl.} {\bf 91}, 339--363 (2009)

\bibitem{CMN2} {\sc A. Cesaroni, C. B. Muratov, M. Novaga}: 
Asymptotic behavior of attractors for inhomogeneous Allen-Cahn
equations. RIMS Kokyuroku, to appear.

\bibitem{cn} {\sc A. Cesaroni, M. Novaga}: 
Long-time behavior of the mean
curvature flow with periodic forcing. 
{\it Comm. Partial Differential
Equations} {\bf 38}, 780-801 (2013)

\bibitem{chen} {\sc X. Chen}: 
Generation and propagation of
interfaces in reaction-diffusion equations: 
{\it J. Differential Equations} {\bf 96},  116--141 (1992)

\bibitem{dms} {\sc P. De Mottoni, M. Schatzman}: 
Geometrical evolution of
developed interfaces. 
{\it Trans. Amer. Math. Soc.} {\bf 347},
1533--1589 (1995)

\bibitem{ess} {\sc L. C. Evans, H. M. Soner, P. E. Souganidis}: 
Phase transitions and generalized motion by mean curvature.
{\it
Comm. Pure Appl. Math.} {\bf 45}, 1097--1123 (1992)

\bibitem{fv} {\sc A. Farina, E. Valdinoci}: 
Geometry of
quasiminimal phase transitions. 
{\it Calc. Var. Partial
Differential Equations} {\bf 33}, 1--35 (2008)

\bibitem{fife:jcp76} {\sc P. C. Fife}: 
Pattern formation in
  reacting and diffusing systems. 
{\it J. Chem. Phys.} {\bf
  64}, 554--563 (1976)

\bibitem{fife} {\sc P. C. Fife}:
{\it Dynamics of Internal Layers and Diffusive
Interfaces.} Society for Industrial and Applied Mathematics,
Philadelphia (1988)

\bibitem{fifemcl} {\sc P. C. Fife, J. B. McLeod}: 
The approach of solutions
of nonlinear diffusion equations to travelling front
solutions.
{\it Arch. Ration. Mech. Anal.} {\bf 65}, 335--361 (1977)

\bibitem{gt} {\sc D. Gilbarg, N. S. Trudinger}:
{\it
Elliptic Partial Differential Equations of Second Order.}
Springer-Verlag, Berlin (1983)

\bibitem{giustibook} {\sc E. Giusti}:
{\it
Minimal surfaces and functions of bounded variation.}
 Birkh\"auser Verlag, Basel (1984)

\bibitem{gr1} {\sc M. Gr\"uter}: 
Boundary regularity for solutions of a
  partitioning problem. 
{\it Arch. Rational Mech. Anal.} {\bf 97}(3),
  261--270 (1987)

\bibitem{gr2} {\sc M. Gr\"uter}: 
Optimal regularity for codimension one
  minimal surfaces with a free boundary. 
{\it Manuscripta Math.} {\bf
    58}(3), 295--343 (1987)

\bibitem{gr3} {\sc M. Gr\"uter}: 
Regularity results for minimizing currents
  with a free boundary. 
{\it J. Reine Angew. Math.} {\bf 375/376}, 307--325
  (1987)

\bibitem{ilmanen93} {\sc T.~Ilmanen}: 
Convergence of the {A}llen-{C}ahn
equation to {B}rakke's motion by mean curvature. 
{\it J. Differential
Geom.} {\bf 38}, 417--461 (1993)

\bibitem{jones83} {\sc C. K. R. T. Jones}: 
Asymptotic behaviour
of a reaction-diffusion equation in higher space dimensions.
{\it Rocky Mountain J. Math.} {\bf 13}, 355--364 (1983)

\bibitem{lmn} {\sc M. Lucia, C. B. Muratov, M. Novaga}:  
Existence of traveling wave solutions for Ginzburg-Landau-type
problems in infinite cylinders.
{\it Arch. Rat. Mech. Anal.} {\bf 188}, 475--508 (2008)

\bibitem{maggi} {\sc F. Maggi}: {\it Sets of Finite Perimeter and Geometric
Variational Problems.}  Cambridge Studies in Advanced Mathematics,
Cambridge (2012)


\bibitem{momo}
{\sc L. Modica, S. Mortola}: 
Un esempio di $\Gamma$-convergenza. 
{\it  Bollettino U.M.I.} {\bf 14-B}, 285--299 (1977)

\bibitem{mugnai11} {\sc L.~Mugnai, M.~R{\"o}ger}: 
Convergence of perturbed
{A}llen-{C}ahn equations to forced mean curvature flow. 
{\it Indiana
Univ. Math. J.} {\bf 60}, 41--75 (2011)

\bibitem{mu} {\sc C. B. Muratov}: 
A global variational structure
and propagation of disturbances in reaction-diffusion systems of
gradient type.
{\it Discrete Cont. Dyn. Syst. B} {\bf 4}, 867--892 (2004)

\bibitem{mn1} {\sc C. B. Muratov, M. Novaga}: 
Front propagation
in infinite cylinders I. A variational approach.
{\it  Comm. Math. Sci.} {\bf 6}, 799--826 (2008)

\bibitem{mn2} {\sc  C. B. Muratov, M. Novaga}: 
Front
propagation in infinite cylinders. II. The sharp reaction zone
limit.  
{\it Calc. Var. Partial Differential Equations} {\bf
31}, 521--547 (2008)

\bibitem{mn3} {\sc C. B. Muratov, M. Novaga}: 
Global
exponential convergence to variational traveling waves in cylinders.
{\it SIAM J. Math. Anal.} {\bf 44}, 293--315 (2012)

\bibitem{nv} {\sc M. Novaga, E. Valdinoci}: 
The geometry of mesoscopic phase
  transition interfaces.  
{\it Discrete Cont. Dyn. Syst. A} {\bf 19},
  777--798 (2008)
	
\bibitem{pw} {\sc M.~H. Protter, H.~F. Weinberger}: 
{\it Maximum
    principles in differential equations.} Springer-Verlag, New York (1984)

\bibitem{rubinstein89} {\sc J. Rubinstein, P. Sternberg, J. B. Keller}:
Fast reaction, slow diffusion, and curve shortening.
{\it
SIAM J. Appl. Math.} {\bf 49},  116--133 (1989)

\bibitem{taylor} {\sc J. E. Taylor}: 
Boundary regularity for
  solutions to various capillarity and free boundary
  problems.
{\it Comm. Partial Differential Equations} {\bf 2}, 323-357 (1977)

\bibitem{vega2} {\sc J. M. Vega}: 
The asymptotic behavior of the
  solutions of some semilinear elliptic equations in cylindrical
  domains.
{\it J. Differential Equations} {\bf 102}, 119--152 (1993)

\bibitem{xin} {\sc J. Xin}: 
Front propagation in heterogeneous
  media. 
{\it  SIAM Review} {\bf 42}, 161--230 (2000)

\end{thebibliography}
% \bibliographystyle{plain}

\end{document}